\RequirePackage{fix-cm}
\documentclass[smallextended]{svjour3}       
\smartqed  
\usepackage[all]{xy}
\usepackage{subfigure}
\usepackage{graphicx}
\usepackage{amsmath,amssymb,algorithm,rotating,amsfonts,algpseudocode}
\usepackage[colorlinks=true,linkcolor=red,citecolor=blue]{hyperref}
\usepackage{color}
\usepackage{cite}

\def\r{\textcolor{black}}

 \journalname{}

\newcommand{\co}{\mathop{\mathrm{conv}}}
\def\cl{\mbox{\rm cl}\,}
\def\Limsup{\mathop{{\rm Lim}\,{\rm sup}}}

\def\co{\mbox{\rm co}\,}

\def\Int{\mbox{\rm int}\,}
\def\dom{\mbox{\rm dom}\,}

\begin{document}

\title{ Optimization Problems with Difference of Tangentially Convex Functions under Uncertainty }
\titlerunning{Optimization Problems with DTC Functions under Uncertainty }        

\author{Feryal Mashkoorzadeh$^{1,2,3}$   \and Nooshin Movahedian$^{2}$ 
}

\institute{ \at
               $^{1}$  Iran National Science Foundation (INSF), P.O. Box: 15875-3939, Tehran, Iran \\
               $^{2}$  Department of Applied Mathematics and Computer Sciences, Faculty of Mathematics and Statistics, University of Isfahan, P. O. Box: 81745-163, Isfahan, Iran \\
              $^{3}$  School of Mathematics, Institute for Research in Fundamental Sciences (IPM),
P.O. Box: 19395-5746, Tehran, Iran\\
              \email{f.mashkoorzadeh@sci.ui.ac.ir; n.movahedian@sci.ui.ac.ir }
}

\date{Received: date / Accepted: date}

\maketitle

\begin{abstract}
This paper investigates a specific class of nonsmooth nonconvex optimization problems in the face of data uncertainty, namely, robust optimization problems, where the given objective function can be expressed as a difference of two tangentially convex (DTC) functions. 
More precisely, we develop a range of nonsmooth calculus rules to establish relationships between Fr\'{e}chet and limiting subdifferentials for a particular maximum function and the tangential subdifferential of its constituent functions. Subsequently, we derive optimality conditions for problems involving DTC functions, employing generalized constraint qualifications within the framework of the tangential subdifferential concept.
Several illustrative examples are presented to demonstrate the obtained results.

 
\end{abstract}

\section{Introduction}
Real-life optimization problems often encounter data uncertainty due to incomplete information, prediction errors, and measurement errors.  
Numerous fields, such as logistics, finance, water management, energy management, machine learning, and major emergencies like COVID-19, frequently involve problems with uncertain data. 
As a result, significant attention has been devoted to these types of optimization problems from both theoretical and practical perspectives over the past two decades (see, e.g., \cite{Garcia}). 

The primary objective of this study is to investigate optimality conditions for a specific class of optimization problems that involve data uncertainty in the constituent functions:
\begin{eqnarray*}
  \begin{array}{lr}
 \min\hspace{.9cm}
 \psi_0(x)\hspace{1cm}(RP_S) \\
 \hspace{.1cm}\mbox{s.t.}\hspace{1cm} x\in S,
  \end{array}
\end{eqnarray*}
where $\psi_0:\mathbb{R}^n\rightarrow\mathbb{R}\cup\{+\infty\}$ and $S$ is a nonempty closed subset of $\mathbb{R}^n.$ 

Robust optimization is a computationally powerful deterministic approach applicable to many classes of optimization problems involving data uncertainty. The goal of robust optimization is to find a worst-case solution that satisfies all possible realizations of the constraints, thereby immunizing the optimization problem against uncertain parameters. There are several methods for applying robust optimization, and the choice of method typically depends on the problem being solved.

Over the last three decades, numerous studies have examined nonsmooth and nonconvex programming problems, particularly DC optimization problems, where the objective and/or constraint functions are expressed as the difference of two convex functions. Indeed, DC programming and the DC Algorithm (DCA) are widely recognized as essential tools for addressing nonsmooth and nonconvex programming problems. 
Recent work \cite{p3} introduced a broader class of nonconvex, nonsmooth problems where functions are expressed as differences of two tangentially convex (DTC) functions. Motivated by these advances, we derive optimality conditions for both unconstrained and constrained optimization problems under data uncertainty. Our approach combines robust optimization techniques with the tangential subdifferential concept. Specifically, we:
\begin{itemize}
\item[1.] Analyze a special nonsmooth maximum function, proving its DTC property using variational analysis.
\item[2.] Explore relationships between the Fr\'{e}chet subdifferential, limiting subdifferential, and tangential subdifferential for constituent functions.
\item[3.] Employ generalized error bounds and Abadie constraint qualifications within the tangential subdifferential framework to establish optimality conditions for the robust DTC problem $RP_S.$
\end{itemize} Unlike prior studies, we relax the convexity assumption on uncertain sets and feasible sets, as well as the concavity requirement for functions with respect to uncertain parameters. Our results generalize existing work by encompassing cases where these stronger assumptions hold, thereby deriving sharper local optimality conditions under weaker constraint qualifications.

The remainder of the paper is organized as follows. Section \ref{pre} contains the basic definitions and preliminary results of convex and nonsmooth analysis needed in the article.
Several results in nonsmooth analysis that characterize some  suitable properties for a specific maximum function are established in Section \ref{max}. In Section \ref{CQ}, two robust constraint qualifications are introduced based on the concept of tangential subdifferential. Furthermore, the relationship between them is exploited. Finally in Section \ref{OPT}, new optimality conditions for problem $RP_S$ are proved for two special types of objective functions. 
Several examples are given to illustrate the results of the paper.

\section{Preliminaries}\label{pre}
In this section, we recall some basic definitions and results from nonsmooth analysis needed throughout the paper. For more details we refer the reader to \cite{Bazaraa,Borwein,fclarke,Clarke}.

Throughout the paper, we use the following standard notations.
Suppose that $\mathbb{R}^n$ signifies  Euclidean space whose norm is denoted by $||.||.$ The inner product of two vectors $x,y\in\mathbb{R}^n$ is shown by $\langle x,y\rangle=\sum_{i=1}^nx_iy_i.$ 
By $\mathbb{B},$ we denote the closed unit ball centered at the origin of $\mathbb{R}^n,$ while $\mathbb{B}_{\delta}(x)$ stands for the closed ball centered at $x$ with radius $\delta>0.$
Let $S$ be a given subset of $\mathbb{R}^n,$ the distance function $d_S:\mathbb{R}^n\rightarrow\mathbb{R}_+$ is defined by $d_S(x):=\underset{y\in S}{\inf}~||y-x||.$ The negative polar cone, the closure and convex hull of $S,$ are denoted, respectively, by $S^-,~\cl S$ and $\co S.$

If $S$ is closed 
and $\bar{x}\in S,$ then
the contingent cone of $S$  and the Fr\'{e}chet normal cone to $S$ at $\bar{x}$ are defined, respectively, by
\[\begin{array}{ll}
T(\bar{x};S):=\Big\{d\in\mathbb{R}^n|\ \exists t_k\downarrow0,\ \exists
d_k\rightarrow d,\ \mbox{s.t.}\  \bar{x}+t_kd_k\in S,
 \ \forall k\Big\}.
 \end{array}\]
and
\[\hat{ N}(\bar{x};S):=(T(\bar{x};S))^-=\Big\{\xi\in\mathbb{R}^n|\ \langle\xi,d\rangle\leq0,\ \forall d\in T(\bar{x};S)\Big\}.
\]
Consider $f:\mathbb{R}^n\rightarrow\mathbb{R}\cup\{+\infty\}$ and  $\bar{x}\in\dom f:=\{x\in\mathbb{R}^n\ |\ f(x)<\infty \}.$
 $f$ is lower semicontinuous (l.s.c.)  at $\bar{x}$ provided that
 $
 \liminf_{x\rightarrow\bar{x}}f(x)\geq f(\bar{x}).
 $
Complementary to lower semicontinuity  of $f,$ we say that $f$ is upper semi continuous (u.s.c.)  at $\bar{x}$ if  $-f$ is l.s.c. at this point.
 
The lower and upper Dini derivatives of $f$ at $\bar{x}$ in the direction $d\in\mathbb{R}^n$ are defined, respectively, by
\begin{eqnarray*}
\begin{array}{ll}
f^-(\bar{x};d):=\underset{t\downarrow0}{\liminf}\dfrac{f(\bar{x}+td)-f(\bar{x})}{t},\\
f^+(\bar{x};d):=\underset{t\downarrow0}{\limsup}\dfrac{f(\bar{x}+td)-f(\bar{x})}{t}.
\end{array}
\end{eqnarray*}

The directional derivative of $f$ at $\bar{x}$ in the direction $d$ is given by
\begin{eqnarray}\label{11111}
f'(\bar{x};d):=\lim_{t\downarrow0}\frac{f(\bar{x}+td)-f(\bar{x})}{t},
\end{eqnarray}
when the limit in (\ref{11111}) exists.
Following \cite{boris3}, if $f$ is l.s.c. around $\bar{x}\in\dom f$,
then the presubdifferential or Fr\'{e}chet subdifferential of f at $\bar{x}$ is defined by
\begin{equation*}
\hat{\partial}f(\bar{x}):=\{\xi\in\mathbb{R}^n~|~\underset{x\rightarrow \bar{x}}{\liminf}~\frac{f(x)-f(\bar{x})-\langle\xi,x-\bar{x}\rangle}{||x-\bar{x}||}\geq0
\},
\end{equation*}
and the limiting or Mordukhovich subdifferential of $f$ at $\bar{x}$ is given by 
\begin{equation}\label{600}
\partial_L f(\bar{x}):=\underset{x\overset{f}{\rightarrow}\bar{x}}{\Limsup}~ \hat{\partial} f(x),
\end{equation}
where  the symbol $x\overset{f}{\rightarrow}\bar{x}$ signifies that $x\rightarrow\bar{x}$ with $f(x)\rightarrow f(\bar{x}).$ Moreover, for any given $\bar{x}\in S,$ the Mordukhovich normal cone is defined by
\begin{equation*}
N_L(\bar{x};S):=\underset{x\rightarrow\bar{x}}{\Limsup}~\hat{N}(\bar{x};S).
\end{equation*}

The following definition presents a concept of a class of functions that was introduced by Pshenichnyi \cite{Pschen} and is  called ``tangentially convex" by Lemar\'{e}chal \cite{lemar}.
\begin{definition}
A function $f:\mathbb{R}^n\rightarrow\mathbb{R}\cup\{+\infty\}$ is said to be tangentially convex at $\bar{x}  \in\dom f$
if for each $d\in\mathbb{R}^n$ the limit in (\ref{11111}) exists, is finite and the function $d\longmapsto f'(\bar{x};d)$ is convex.
\end{definition}
Since $d\longmapsto f'(\bar{x};d)$ is generally positively homogeneous, it follows that for each tangentially convex function, it is sublinear \cite{Borwein,Legaz}.

Tangentially convex functions include several known classes of nonconvex and nondifferentiable
functions, such as regular functions in the sense of Clarke and Michel-Penot, G\^{a}teaux differentiable functions, and also convex functions with open domain. This class is closed under addition and multiplication by positive scalars. Moreover, the product of two nonnegative tangentially convex functions is tangentially convex. 

The following definition, presents the notion of tangential subdifferential
that was defined in \cite{Pschen}.
\begin{definition}
Let $f:\mathbb{R}^n\rightarrow\mathbb{R}\cup\{+\infty\}$ be a tangentially convex function at $\bar{x}\in\dom f.$
The tangential subdifferential of  $f$  at $\bar{x}$ given by  the set
\[\partial_Tf(\bar{x}):=\Big\{\xi\in\mathbb{R}^n\ |\  \langle \xi,d\rangle\leq f'(\bar{x};d),\  \forall d\in\mathbb{R}^n\Big\}.\]
\end{definition}

It is worth to mention that 
$\partial_Tf(\bar{x})$ is a nonempty closed convex subset of $\mathbb{R}^n.$ It is also easily seen that $f'(\bar{x};.)$ is the support function of $\partial_Tf(\bar{x}),$ that is, for each $ d\in\mathbb{R}^n$ one has
$$f'(\bar{x};d)=\hspace{-.2cm}\max_{\xi\in{\partial_Tf(\bar{x})}}\langle \xi,d\rangle.$$
Obviously, if $f$ is a convex function then
$\partial_Tf(\bar{x})$ reduces to the classical convex subdifferential $\partial f(\bar{x})$ \cite{Borwein}.
The following lemma, provides us with the relationship between F\'{e}chet subdifferential and tangential subdifferential.  
\begin{lemma}\label{lem2}
If $f:\mathbb{R}^n\rightarrow\mathbb{R}\cup\{+\infty\}$ is l.s.c. and tangentially convex at $\bar{x}\in\mathbb{R}^n,$ then $\hat{\partial}f(\bar{x})\subseteq\partial_Tf(\bar{x}).$ The converse holds provided that $f$ is assumed to be Lipschitz near $\bar{x}.$ 
\end{lemma}
\begin{proof}
If $\hat{\partial}f(\bar{x})\neq\emptyset,$ nothing remains to prove. Suppose that $\xi\in\hat{\partial}f(\bar{x}).$ Then one has
\[
\underset{x\rightarrow\bar{x}}{\liminf}~\frac{f(x)-f(\bar{x})-\langle\xi,x-\bar{x}\rangle}{||x-\bar{x}||}\geq0.
\]
Take arbitrary $u\in\mathbb{R}^n\setminus\{0\}$ and consider the sequence $t_k\downarrow0.$ The above immediately implies that 
\[
\frac{1}{||u||}f'(\bar{x};u)=\underset{k\rightarrow\infty}{\lim}~\frac{f(\bar{x}+t_ku)-f(\bar{x})}{t_k||u||}\geq\langle\xi,\frac{u}{||u||}\rangle.
\] 
Hence we deduce that $\xi\in\partial_Tf(\bar{x}).$ To complete the proof of Lemma, suppose that $f$ is Lipschitz near $\bar{x}$ and take $\xi\in\partial_Tf(\bar{x}).$ Fix an arbitrary sequence $x_k\rightarrow\bar{x}$ such that $x_k\neq\bar{x}$ and define $u_k:=\frac{x_k-\bar{x}}{||x_k-\bar{x}||}$ and $t_k:=||x_k-\bar{x}||.$ Then one can assume without loss of generality that $\underset{k\rightarrow\infty}{\lim} v_k=v\in\mathbb{R}^n.$ Since $\xi\in\partial_Tf(\bar{x})$ and $f$ is Lipschitz at $\bar{x},$ one has
\begin{align*}
\underset{k\rightarrow\infty}{\lim}\langle\xi,\frac{x_k-\bar{x}}{||x_k-\bar{x}||}\rangle & = \langle\xi,v\rangle\\
 & \leq f'(\bar{x};v)\\
 & =\underset{k\rightarrow\infty}{\lim}~\frac{f(\bar{x}+t_kv)-f(\bar{x})}{t_k}\\
 & =\underset{k\rightarrow\infty}{\lim}~\frac{f(\bar{x}+t_kv_k)-f(\bar{x})}{t_k}\\
 & =\underset{k\rightarrow\infty}{\lim}~\frac{f(x_k)-f(\bar{x})}{||x_k-\bar{x}||}.
\end{align*}
The above especially implies that $\xi\in\hat{\partial}f(\bar{x})$ and completes the proof of Lemma.
\qed
\end{proof}
The following example shows that if $f$ is not Lipschitz then the reverse inclusion in Lemma \ref{lem2} may not hold. 
\begin{example}
Consider the function $f:\mathbb{R}^2\rightarrow\mathbb{R}$ defined by
\[
f(x_1,x_2):=\left\{
\begin{array}{ll}
\frac{x_1^3}{x_2}+x_1, &\  x_2\neq0\\
x_1,  & \ x_2=0,
\end{array} \right.
\]
We claim that $f$ is not locally Lipschitz at $\bar{x}=(0,0).$ If we consider the sequence $x_k:=(x_{1k},x_{2k})=(-\frac{1}{k},\frac{1}{k^4})\rightarrow\bar{x},$ then one has for each $k$
\[
|f(x_k)-f(\bar{x})|=k+\frac{1}{k}\geq k||x_k-\bar{x}||.
\]
However, we can show that $f$ is G\^{a}teaux differentiable at $\bar{x}.$ By taking $d=(d_1,d_2)\in\mathbb{R}^2$ such that $d_2\neq0,$ we get
\[
\underset{t\downarrow0}{\lim}\frac{f(\bar{x}+td)-f(\bar{x})}{t}=\underset{t\downarrow0}{\lim}\frac{td_1^3}{d_2}+d_1=d_1,
\]
which immediately implies that for all $d\in\mathbb{R}^2,$
$
f'(\bar{x};d)=d_1.
$
Hence $f$ is G\^{a}teaux differentiable at $\bar{x},$ and particularly, tangentially convex at this point with $\partial_Tf(\bar{x})=\{(1,0)\}.$ Meanwhile, $\hat{\partial}f(\bar{x})=\emptyset.$ Suppose on the contrary that $\xi\in\hat{\partial}f(\bar{x})$ and consider the sequence $x_k:=(x_{1k},x_{2k})=(-\frac{1}{k},\frac{1}{k^4})\rightarrow\bar{x}.$ Then one has 
\begin{align*}
0 & \leq\underset{k\rightarrow\infty}{\lim}~\frac{f(x_k)-f(\bar{x})-\langle\xi,x_k-\bar{x}\rangle}{||x_k-\bar{x}||}\\
 & =\underset{k\rightarrow\infty}{\lim}~\frac{-k^2+1+\xi_1-\frac{\xi_2}{k^3}}{\sqrt{1+\frac{1}{k^6}}}\\
 & =-\infty.
\end{align*}
 which is a contradiction. Hence, $\emptyset=\hat{\partial}f(\bar{x})\subsetneq\partial_Tf(\bar{x}).$
\end{example}
The following lemma recalling from \cite{MMN1} shows that the tangential subdifferential of a tangentially convex function is a bounded subset of $\mathbb{R}^n.$
\begin{lemma}
Suppose that $f:\mathbb{R}^n\rightarrow\mathbb{R}\cup\{+\infty\}$ is tangentially convex at $\bar{x}\in\dom f.$ Then $\partial_Tf(\bar{x})$ is bounded.
\end{lemma}
\begin{remark}\label{lltc}
Let $f:\mathbb{R}^n\rightarrow\mathbb{R}\cup\{+\infty\}$ be a tangentially convex function at $\bar{x}\in\dom f.$ Then defining the convex function  $\tilde{f}(d):=f'(\bar{x};d),$  one can observe that for each $d\in\mathbb{R}^n$, $ \tilde{f}'(0;d)=\tilde{f}(d)=f'(\bar{x};d).$ Thus it is easy to show that $\partial\tilde{f}(0)=\partial_Tf(\bar{x}).$

\end{remark}
Finally in this section, let us recall the notion of Fenchel conjugate  from \cite{Borwein}.
The Fenchel conjugate of the function $f:\mathbb{R}^n\rightarrow\mathbb{R}\cup\{+\infty\}$ is defined  by $f^*:\mathbb{R}^n\rightarrow\mathbb{R}\cup\{\pm\infty\}$ with
$  f^*(\xi):=\sup_{x\in\mathbb{R}^n}\{\langle\xi,x\rangle-f(x)\}$, $\xi\in\mathbb{R}^n. $
Obviously, $f^*$ is a convex function and never takes the value $-\infty$ provided that $\dom f\neq\emptyset.$
 According to the Fenchel-Young inequality \cite[Proposition 3.3.4]{Borwein}, one has for each $\xi\in\mathbb{R}^n$ and all $x\in\dom f,$
\begin{equation}\label{eqf1}
  f(x)+f^*(\xi)\geq\langle\xi,x\rangle.
\end{equation}
 Moreover, the convexity of $f$ implies the equality in (\ref{eqf1}) if and only if $\xi\in\partial f(x).$

\section{Max Function}\label{max}
In this section, we try to verify some facts about a certain ``max function" defined this way.

Consider a compact subset $V\subseteq\mathbb{R}^q,$ 
 and the functions
 $g,h:\mathbb{R}^n\times V\rightarrow\mathbb{R}\cup\{+\infty\}.$ Let the function
 $\psi:\mathbb{R}^n\rightarrow\mathbb{R}\cup\{+\infty\}$
  and the multifunction $V:\mathbb{R}^n\rightrightarrows \mathbb{R}^q $ be defined respectively by
 \begin{equation}\label{psi1}
 \psi(x):=\max_{v\in V}\ g(x,v)-h(x,v),
 \end{equation}
 \begin{equation}\label{eqv1}
  V(x):=\Big\{v\in V\ |\ \psi(x)=g(x,v)-h(x,v)\Big\}.
 \end{equation}
From now on, we put some assumptions on $g,~h$ as follows:
\begin{enumerate}
   \item [$(A_1)$] The functions $g$ and $h$ are u.s.c. and l.s.c., respectively, at each $(x,v)\in\mathbb{R}^n\times V.$
    \item [$(A_2)$] $g$ and $h$ are locally Lipschitz in $x,$ uniformly for $v\in V;$ i.e., for each $x\in\mathbb{R}^n,$ one can find an open neighbourhood $N(x)$ of $x$ and a positive number $K$ such that for every $y,z\in N(x)$ and $v\in V:$ 
        \[
       \max\Big\{ |(g(y,v)-g(z,v)|,|(h(y,v)-h(z,v)|\Big\}\leq K||y-z||.
        \]
   \item [$(A_3)$] $g$ and $h$ are  tangentially convex functions with respect to $x;$ i.e., $g'_x(x,v;d)$ and $h'_x(x,v;d)$ which are the directional derivative of $k$ and $h$ with respect to $x,$ respectively, exist, are finite for each $(x,v)\in\mathbb{R}^n\times V,$ and are convex functions of $d\in\mathbb{R}^n.$ 
   \item [$(A_4)$] The mappings $v\mapsto g'_x(x,v;d)$ and  $v\mapsto h'_x(x,v;d)$ are u.s.c. and l.s.c., respectively, at each $v\in V.$
 \end{enumerate}

The next theorem investigates a connection between the directional derivatives of the functions $\psi$ and $g, h.$  Indeed, it is a nonsmooth version of Danskin's theorem for max-functions (see \cite{fclarke,p2} for more details). 

\begin{theorem}\label{l1}
Consider the max function $\psi$  in (\ref{psi1}). The directional derivative  of $\psi$ at $x\in\mathbb{R}^n$ exists and is obtained by 
 \begin{equation}\label{reql1}
 \psi'(x;d)=\max_{v\in V(x)}g'_x(x,v;d)-h'_x(x,v;d),\ \forall d\in\mathbb{R}^n.
 \end{equation}
 \end{theorem}
 \begin{proof}
  It is easily seen that the function $g-h$ satisfies all the assumptions of \cite[Theorem 1]{p2}. Thus it follows immediately that 
 \begin{equation*}\label{}
 \psi'(x;d)=\max_{v\in V(x)}(g-h)'_x(x,v;d),
 \end{equation*} 
for each $d\in\mathbb{R}^n.$ Now according to the definition of the directional derivative for $g-h$ we get (\ref{reql1}), and the proof is completed.
  \qed
 \end{proof}
 The next proposition provides an important property for the max function $\psi'$ defined in (\ref{reql1}).
 \begin{proposition}\label{rp002}
 Assume that the functions $g'_x(.,.;d)$ and $h'_x(.,.;d)$ are nonnegative for each $d\in\mathbb{R}^n.$ Then $\psi'(.;d)$ in (\ref{reql1}) is a DC function.
 \end{proposition}
 \begin{proof}
 Using Theorem \ref{l1} together with \cite[Proposition 2.2]{Correa}, 
 one can easily find that $\psi'(.,d)$ is a DC function for each $d\in\mathbb{R}^n$.
 \qed
 \end{proof}
The next example illustrates the validity of Theorem \ref{l1}  and Proposition \ref{rp002}.
\begin{example}
Consider the following max function
  \[
    \displaystyle \psi(x)=\underset{v\in V}{\max}~(\cos v_1v_2)|x_1|-e^{\sin{|x_2|}}+1,
    \]
where $V:=\{v=(v_1,v_2)\in\mathbb{R}^2\ |\ v_1,v_2\in[0,\frac{\pi}{2}]\}.$ Defining  $g(x,v):=(\cos v_1v_2)|x_1|$ and $h(x,v):=e^{\sin{|x_2|}}-1,$ it is clear that  $g$ and $h$ are satisfied assumptions $A_1-A_4$   at $\bar{x}=(0,0)\in\mathbb{R}^2.$
A simple computation gives us 
$\displaystyle \psi(x)=|x_1|-e^{\sin{|x_2|}}+1$
and
\[
V(x)=
\left\{
\begin{array}{ll}
V, & \ x_1=0\\
\{(v_1,v_2)\in V~|~v_1v_2=0\},  &\ x_1\neq0.
\end{array} \right.
\]
Therefore, for each $v\in V(\bar{x}),~g'_x(\bar{x},v;d)=|d_1|,~h'_x(\bar{x},v;d)=|d_2|.$ Further, it follows that
$\psi'(\bar{x};d)=|d_1|-|d_2|,$ for each $d=(d_1,d_2)\in\mathbb{R}^2,~ v\in V(\bar{x})$ which implies (\ref{reql1}).  
 Moreover,  $g'_x(\bar{x},v;d)$  and $h'_x(\bar{x},v;d)$ are nonnegative at each $d\in\mathbb{R}^2$ and due to Proposition \ref{rp002}, $\psi'(\bar{x};d)$ is a DC function w.r.t. $d.$
\end{example} 
In the following result, we show a criterion to ensure that the supremum of DTC functions is still a DTC function.
For this purpose, it is necessary to take another assumption in addition to the previous assumptions $A_1-A_4.$
\begin{enumerate}
   \item [$(A_5)$] The functions $g$ and $h$ are all nonnegative at each $(x,v)\in\mathbb{R}^n\times V.$
\end{enumerate}
We also denote the family of finite subsets of $V$ with $\mathcal{P}(V)$ and define the following functions
\[
g_F(x):=\underset{\nu\in F}{\max}~\big(g_\nu(x)+\sum_{\omega\in F\setminus\{\nu\}}h_\omega(x)\big),~~~~~h_F(x):=\sum_{\omega\in F}h_s(x),
\]
where $F\in\mathcal{P}(V).$ Further, we define the following supremum functions
\begin{equation}\label{sup1}
g(x):=\underset{F\in\mathcal{P}(V)}{\sup} g_F(x),~~~~
h(x):=\underset{F\in\mathcal{P}(V)}{\sup} h_F(x).
\end{equation}
 According to \cite[Theorem 1]{p2}, 
it is easy to check that $g_F,~h_F$ and hence $g,~h$ are all tangentially convex functions at $x\in\mathbb{R}^n.$ 
 
\begin{proposition}\label{DTCproperty1}
Consider a family of DTC functions $\psi_\nu=g_\nu-h_\nu,$ where $g_\nu$ and $h_\nu$ satisfy  assumptions $A_1-A_5.$ Then, $\psi:=\underset{\nu\in V}{\sup}~ \psi_\nu$ is a DTC function over $\dom h.$ Further, $\psi(x)=g(x)-h(x)$ for all $x\in\dom h.$
 \end{proposition} 
 \begin{proof}
  It is clear that for each $F_1\subseteq F_2,$
  \begin{equation}\label{sup2}
  \sum_{\omega\in {F_1}}h_\omega(x)\leq\sum_{\omega\in {F_2}}h_\omega(x).
\end{equation}   
For given $F\in\mathcal{P}(V)$ and $x\in\mathbb{R}^n,$ we denote by $\nu(F,x)$ an index in $F$ such that
\[
g_F(x)=g_{\nu(F,x)}(x)+\sum_{\omega\in F\setminus\{\nu(F,x)\}}h_\omega(x).
\]
For any $\nu\in V$ and $F\in\mathcal{P}(V)$ we have
\begin{align*}
\psi_\nu(x)=g_\nu(x)-h_\nu(x) & =g_\nu(x)+\sum_{\omega\in{F\setminus \{\nu\}}}h_\omega(x)-\sum_{\omega\in {F\cup \{\nu\}}}h_\omega(x)\\
& \leq \underset{\nu\in F}{\max}~\Big(g_\nu(x)+\sum_{\omega\in{F\setminus \{\nu\}}}h_\omega(x)\Big)-\sum_{\omega\in {F\cup \{\nu\}}}h_\omega(x)\\
& \leq g_{F\cup\{\nu\}}-h_{F\cup\{\nu\}}.
\end{align*}
Now, fix $x\in\dom h.$ Then for each arbitrary $\varepsilon>0,$ one can find some $F\in\mathcal{P}(V)$ such that $h(x)-\varepsilon\leq\sum_{\omega\in {F}}h_\omega(x)\leq h_{F\cup \{\nu\}}(x),$ where the last inequality is due to $(\ref{sup2}).$ Hence, for any $\nu\in V,$
\[
\psi_\nu(x)\leq g_{F\cup\{\nu\}}-h_{F\cup\{\nu\}}\leq g(x)-h(x)+\varepsilon.
\]
Since $\varepsilon>0$ and $\nu\in V$ were arbitrarily chosen, we get $\psi(x)\leq g(x)-h(x).$

To prove the converse inequality, we assume that $g(x)\in\mathbb{R}.$ Obviously, if $g(x)=+\infty,$ then the equality holds trivially. Now, suppose that there exists $F\in\mathcal{P}(V)$ such that $g(x)\leq g_F(x)+\varepsilon.$ Hence,
\begin{align*}
g(x)-h(x) & \leq g_F(x)-h_F(x)+\varepsilon \\
& =g_{\nu(F,x)}(x)+\sum_{\omega\in{F\setminus \{\nu(F,x)\}}}h_\omega(x)+\varepsilon-h_{F}(x)\\
& =g_{\nu(F,x)}(x)-h_{\nu(F,x)}(x)+\varepsilon\\
& \leq\psi(x)+\varepsilon.
\end{align*}
Since $\varepsilon>0$ is arbitrary chosen, it follows that  $g(x)-h(x)\leq\psi(x).$ Therefore, the supremum function $\psi$ is stated as a difference of two tangentially convex functions $g,~h$  and the proof is completed. 
 \qed 
 \end{proof}
\begin{remark}
It is important to note that the proof of Proposition \ref{DTCproperty1} is similar to the proof of \cite[Proposition 2.2]{Correa}. However, it should be noted that in Proposition \ref{DTCproperty1}, the functions are assumed to be tangentially convex which is clearly  weaker than convexity assumption in \cite[Proposition 2.2]{Correa}.

\end{remark}
In the following, we remind the reader about the notion  of tangential subdifferential in the face of data uncertainty (for more details see \cite{p2}). Consider $f:\mathbb{R}^n\times V\rightarrow\mathbb{R}\cup\{+\infty\},$ then the tangential subdifferential of $f$ at $\bar{x}\in\mathbb{R}^n$  w.r.t. its first component is given by 
\begin{equation*}
\partial_T^xf(\bar{x},v):=\Big\{\xi\in\mathbb{R}^n~|~\langle \xi,d\rangle\leq f'_x(\bar{x},v;d),~ \forall d\in\mathbb{R}^n,~v\in V(\bar{x})
\Big\}.
\end{equation*}
The next theorem investigates the relationship between  F\'{e}chet subdifferential of $\psi$ and tangential subdifferential of the constructed functions $g,h.$
\begin{theorem} \label{l2}
Consider the max function $\psi(.)$ defined in (\ref{psi1}). Then the following inclusion is satisfied at each  $\bar{x}\in\mathbb{R}^n.$
 \begin{equation}\label{req5}
   \hat{\partial}\psi(\bar{x})\subseteq\co\bigcup_{v\in V(\bar{x})}\Big(\partial_T^xg(\bar{x},v)-\partial_T^xh(\bar{x},v)\Big).
 \end{equation}
 \end{theorem}
 \begin{proof}
Let $\xi\in\hat{\partial}\psi(\bar{x}),$ then  
 $\underset{x\rightarrow \bar{x}}{\liminf}~\frac{\psi(x)-\psi(\bar{x})-\langle\xi,x-\bar{x}\rangle}{||x-\bar{x}||}\geq0.
$
Due to Theorem \ref{l1}, $\psi$ is directional differentiable at $\bar{x}.$ Thus taking an arbitrary $\bar{d}\in\mathbb{R}^n\setminus\{0\}$ and considering the sequence $t_k\downarrow0,$ implies that 
\[ 
 \langle\xi,\frac{\bar{d}}{||\bar{d}||}\rangle\leq\underset{k\rightarrow\infty}{\lim}\frac{\psi(\bar{x}+t_k\bar{d})-\psi(\bar{x})}{t_k||\bar{d}||}=\frac{1}{||\bar{d}||}\psi'(\bar{x};\bar{d}).
 \]
The above implies that 
\[
\langle\xi,\bar{d}\rangle\leq\max_{v\in V(\bar{x})}\Big(g'_x(\bar{x},v;\bar{d})-h'_x(\bar{x},v;\bar{d})\Big),
\] 
where the right inequality is due to (\ref{reql1}). 
 For simplicity, define $\tilde{h}(v,\bar{d}):=h'_x(\bar{x},v;\bar{d}).$ Now by using Fenchel-Young equality for an arbitrary $\phi_h^v\in\partial \tilde{h}(v,\bar{d}),$ we obtain
\[
\langle\xi,\bar{d}\rangle\leq\max_{v\in V(\bar{x})}\Big(g'_x(\bar{x},v;\bar{d})-\langle\phi_h^v,\bar{d}\rangle+\tilde{h}^*(\phi_h^v,v)\Big).
\]  
Hence by \cite[Theorem 2]{p2}, it follows that
 \begin{equation}\label{eq500}
 \xi\in\co\underset{v\in V(\bar{x})}{\bigcup}\Big(\partial_T^xg(\bar{x},v)-\phi_h^v\Big)\in\co\underset{v\in V(\bar{x})}{\bigcup}\Big(\partial_T^xg(\bar{x},v)-\partial \tilde{h}(v,\bar{d})\Big).
 \end{equation}
 It is easy to see that for each $\bar{d}\in\mathbb{R}^n,~ \partial \tilde{h}(v,\bar{d})\subseteq\partial_T^xh(\bar{x},v),$ which implies (\ref{req5}) and completes the proof.
\qed
\end{proof}

In the following, we need to verify two auxiliary lemmas that imply  some important properties needed in the investigation of optimality conditions.
\begin{lemma}\label{00631}
There exist a neighbourhood $N(\bar{x})$ of $\bar{x}\in\mathbb{R}^n$ and a positive scalar $K$ such that for each $x\in N(\bar{x}):$ 
\begin{equation*}
 \underset{v\in V}{\bigcup}\partial_T^x(g-h)(x,v)\subseteq 2K\mathbb{B}.
 \end{equation*}
\end{lemma}
\begin{proof}
Due to assumption $A_2,$ one can easily find a neighbourhood $N(\bar{x})$ of $\bar{x}\in\mathbb{R}^n$ and a positive scalar $L$ such that for each $x\in N(\bar{x}),~\underset{v\in V}{\bigcup}\partial_T^xg(x,v)\subseteq K\mathbb{B}$ and $\underset{v\in V}{\bigcup}-\partial_T^xh(x,v)\subseteq K\mathbb{B}.$ Therefore, by the equality  $\partial_T^x(g-h)(\bar{x},\bar{v})=\partial_T^xg(\bar{x},\bar{v})-\partial_T^xh(\bar{x},\bar{v}),$ the proof is straightforward.
\qed
\end{proof}
\begin{lemma}\label{00601}
For a fixed $d\in\mathbb{R}^n,$ the mapping $(\bar{x},\bar{v})\mapsto(g'_x-h'_x)(\bar{x},\bar{v};d)$ is u.s.c.  at each $(\bar{x},\bar{v})\in\mathbb{R}^n\times V(\bar{x}).$
\end{lemma}
\begin{proof}
Assume that for each $k\in \mathbb{N}$ a sequence $\{(x^k,v^k)\}$ converges to $(\bar{x},\bar{v})\in\mathbb{R}^n\times V(\bar{x}).$  Hence we must prove that 
\[
\underset{k\rightarrow\infty}{\limsup}~ (g'_x-h'_x)(x^k,v^k)\leq (g'_x-h'_x)(\bar{x},\bar{v}).
\]
Obviously, it is equivalent to show the u.s.c. property of $\partial_T^x(-h)(\bar{x},\bar{v}).$
Suppose on the contrary that there exist some positive scalar $\varepsilon,$ together with the sequences $(x^k,v^k)\rightarrow(\bar{x},\bar{v})$
and $\eta^k\in\partial_T^x(g-h)(x^k,v^k)$ such that for each $k\in \mathbb{N},~\eta^k\notin\partial_T^x(g-h)(\bar{x},\bar{v})+\varepsilon\mathbb{B}.$
Applying the convex separation theorem, one can easily find a sequence of nonzero vectors $\omega^k\in\mathbb{R}^n$ such that
\[
\langle\eta^k,\omega^k\rangle >(g'_x-h'_x)(\bar{x},\bar{v};\omega^k)+\varepsilon||\omega^k||.
\]
Defining $\tilde{\omega}^k:=\frac{\omega^k}{||\omega^k||},$ and using the positively homogenous property of $(g'_x-h'_x)(\bar{x},\bar{v};.),$ we get for each $k,$
\begin{eqnarray}\label{eq30'}
\langle\eta^k,\tilde{\omega}^k\rangle> (g'_x-h'_x)(\bar{x},\bar{v};\tilde{\omega}^k)+\varepsilon.
\end{eqnarray}
The boundedness of $\{\eta^k\}$ (due to Lemma \ref{00631}) and $\{\tilde{\omega}^k\}$  allows us to assume without loss of generality that $\eta^k\rightarrow\eta$ and $\tilde{\omega}^k\rightarrow\tilde{\omega}\neq0.$
Thus by passing to the limit in (\ref{eq30'}), we get
\begin{eqnarray}\label{47}
\langle\eta,\tilde{\omega}\rangle\geq (g'_x-h'_x)(\bar{x},\bar{v};\tilde{\omega})+\varepsilon>(g'_x-h'_x)(\bar{x},\bar{v};\tilde{\omega}).
\end{eqnarray}
On the other hand, since $\eta^k\in\partial_T^x(g-h)(x^k,v^k)$, one has for each $k,$
\begin{equation}\label{470}
  \langle\eta^k,\tilde{\omega}\rangle\leq (g'_x-h'_x)(x^k,v^k;\tilde{\omega}).
\end{equation}
Next let us show that for a subsequence $\{(x^{k'},v^{k'}\}$ of $\{(x^k,v^k)\},$ we have
\[\lim_{k'\rightarrow+\infty}(g'_x-h'_x)(x^{k'},v^{k'};\tilde{\omega})\leq (g'_x-h'_x)(\bar{x},\bar{v};\tilde{\omega}).\]
Take an arbitrary sequence $t_s\downarrow0$ and consider the double sequence $f_{k,s}$ defined by
\[f_{k,s}:=\frac{(g-h)(x^k+t_s\tilde{\omega},v^k)-(g-h)(x^k,v^k)}{t_s},\quad k,s\in\mathbb{N}.\]
Using assumption $A_2$ together with the fact that $(x^k,v^k)\rightarrow(\bar{x},\bar{v})$ and $t_s\downarrow0,$ one can obtain for sufficiently large $k$ and $s:$
\[|f_{k,s}|=\frac{|(g-h)(x^k+t_s\tilde{\omega},v^k)-(g-h)(x^k,v^k)|}{t_s}\leq K||\tilde{\omega}||.\]
Hence there is a double subsequence $\{k', s'\}\subseteq\mathbb{N}\times\mathbb{N}$ satisfying
\[\lim_{(k',s')\rightarrow+\infty}f_{k',s'}=\bar{f}\in\mathbb{R}.\]
On the other hand, due to assumption $A_1$  for a fixed $s'\in\mathbb{N},$ it follows that
\[\lim_{k'\rightarrow+\infty}\frac{(g-h)(x^{k'}+t_{s'}\tilde{\omega},v^{k'})-(g-h)(x^{k'},v^{k'})}{t_{s'}}
\leq\frac{(g-h)(\bar{x}+t_{s'}\tilde{\omega},\bar{v})-(g-h)(\bar{x},\bar{v})}{t_{s'}}.\]
Moreover, by the tangential convexity of $g$ and $h$ near $\bar{x}$ for a fixed $k'\in\mathbb{N},$ we obtain
\[\lim_{s'\rightarrow+\infty}\frac{(g-h)(x^{k'}+t_{s'}\tilde{\omega},v^{k'})-(g-h)(x^{k'},v^{k'})}{t_{s'}}
=(g'_x-h'_x)(x^{k'},v^{k'};\tilde{\omega}).\]
Applying now the well-known theorem about double and iterated limits of double sequences \cite[Theorem 8.39 ]{Apostol}, we deduce that
\[\begin{array}{ll}
\bar{f}=\underset{(k',s')\rightarrow+\infty}{\lim}f_{k',s'}&= 
\underset{k'\rightarrow+\infty}{\lim}\underset{s'\rightarrow+\infty}{\lim}
\frac{(g-h)(x^{k'}+t_{s'}\tilde{\omega},v^{k'})-(g-h)(x^{k'},v^{k'})}{t_{s'}}\\
 &  =\underset{s'\rightarrow+\infty}{\lim}\underset{k'\rightarrow+\infty}{\lim}\frac{(g-h)(x^{k'}+t_{s'}\tilde{\omega},v^{k'})-(g-h)(x^{k'},v^{k'})}{t_{s'}}\\
 &\leq \underset{s'\rightarrow+\infty}{\lim}\frac{(g-h)(\bar{x}+t_{s'}\tilde{\omega},\bar{v})-(g-h)(\bar{x},\bar{v})}{t_{s'}}\\
& =(g'_x-h'_x)(\bar{x},\bar{v};\tilde{\omega}).
\end{array}\]
Taking limit as $k'\rightarrow+\infty$ in (\ref{470}), gives us
$
\langle\eta,\tilde{\omega}\rangle\leq (g'_x-h'_x)(\bar{x},\bar{v};\tilde{\omega})
$
which contradicts (\ref{47}). Hence $\partial_T^x(g-h)(\bar{x},\bar{v})$ is u.s.c. that implies the u.s.c. property for the mapping $(\bar{x},\bar{v})\mapsto(g'_x-h'_x)(\bar{x},\bar{v};d)$ and completes the proof of lemma.
 \qed
\end{proof}
The last theorem of this section provides the interrelation of the limiting subdifferential of the max function (\ref{psi1}) and the tangential subdifferential of its constructed functions.
\begin{theorem}\label{601}
Consider the max function $\psi(.)$ defined in (\ref{psi1}). Then the limiting subdifferential $\partial_L\psi(.)$ at each $\bar{x}\in\mathbb{R}^n$ satisfies the following inclusion:
 \begin{equation}\label{boM}
   \partial_L\psi(\bar{x})\subseteq\co\bigcup_{v\in V(\bar{x})}\Big(\partial_T^xg(\bar{x},v)-\partial_T^xh(\bar{x},v)\Big).
 \end{equation}
\end{theorem}
\begin{proof}
Pick $\xi\in\partial_L\psi(\bar{x}),$ then there exist sequences  $\xi^k\rightarrow\xi$ and $x^k\rightarrow\bar{x}$ such that $\xi^k\in\hat{\partial}\psi(x^k).$ By (\ref{req5}), it immediately follows that
\[
\xi^k\in\co\bigcup_{v\in V(x^k)}\Big(\partial_T^xg(x^k,v)-\partial_T^xh(x^k,v)\Big).
\]
For each $k,$ consider arbitrary $\xi_i^k\in\partial_T^xg(x^k,v_i^k)-\partial_T^xh(x^k,v_i^k)$ with $v_i^k\in V(x^k)\subseteq V$ such that $\xi^k=\sum_{i=1}^{n+1}\lambda_i^k\xi_i^k,~\sum_{i=1}^{n+1}\lambda_i^k=1$ and $\lambda_i^k\geq0.$ 
In the following, we show that for each $i,~v_i^k\rightarrow v_i\in V(\bar{x})\subseteq V.$
Obviously, $\{v_i^k\}_{k=1}^{\infty}\subseteq V,$ where $V$ is a compact subset of $\mathbb{R}^n.$
Passing to a subsequence, one can assume that for each $i:
\
\underset{k\longrightarrow\infty}{\lim}~v_i^k=v_i\in V.
$

Next, we prove that $v_i\in V(\bar{x}).$ To this end, consider an arbitrary $\hat{v}\in V$ and fix $i.$ Assumption $A_2,$ implies that the mapping $x\mapsto(g-h)(x,\hat{v})$ is continuous. Further, due to assumption $A_1$ and $v_i^k\in V(x^k)$ for all $k,$ we get 
\begin{align*}
 (g-h)(\bar{x},\hat{v})=\underset{k\rightarrow\infty}{\lim}(g-h)(x^k,\hat{v}) & \leq \underset{k\rightarrow\infty}{\limsup}~(g-h)(x^k,v_i^k)\\
 &\leq (g-h)~(\bar{x},v_i).
\end{align*}
Since $\hat{v}$ was chosen arbitrarily, the above inequalities imply that $v_i\in V(\bar{x}),$ and hence $\psi(\bar{x})=g(\bar{x},v_i)-h(\bar{x},v_i).$
On the other hand, Lemma \ref{00631} yields the boundedness property of $\underset{k\in\mathbb{N}}{\bigcup}\Big(\partial_T^xg(x^k,v_i^k)-\partial_T^xh(x^k,v_i^k)\Big).$ Hence by passing to a subsequence, it follows that $\xi_i^k\rightarrow\xi_i.$ We now show that $\xi_i\in\partial_T^xg(\bar{x},v_i)-\partial_T^xh(\bar{x},v_i).$ Since $\xi_i^k\in\partial_T^xg(x^k,v_i^k)-\partial_T^xh(x^k,v_i^k),$  for each $d\in\mathbb{R}^n,$ we have
\[
\langle\xi_i^k,d\rangle\leq g'_x(x^k,v_i^k;d)-h'_x(x^k,v_i^k;d).
\]
Then by passing to the limit and using Lemma \ref{00601}, we obtain: 
\begin{align*}
\langle\xi,d\rangle &\leq\underset{k\rightarrow\infty}{\limsup}~ (g'_x(x^k,v_i^k;d)-h'_x(x^k,v_i^k;d))\\
& \leq g'_x(\bar{x},v_i;d)-h'_x(\bar{x},v_i;d).
 \end{align*}
 Therefore, $\xi_i\in\partial_T^xg(\bar{x},v_i)-\partial_T^xh(\bar{x},v_i).$ 
Furthermore, it is easily seen that $\lambda_i^k\rightarrow\lambda_i$ with $\sum_{i=1}^{n+1}\lambda_i=1,$ due to the boundedness of $\{\lambda_i^k\}.$ 
Finally, putting all the above arguments together, we conclude that
 \begin{align*}
\xi & = \underset{k\rightarrow\infty}{\lim}\xi^k\\
& =\underset{k\rightarrow\infty}{\lim}\sum_{i=1}^{n+1}\lambda_i^k\xi_i^k\\
&=\sum_{i=1}^{n+1}\lambda_i\xi_i\\
&\in\co\bigcup_{v\in V(\bar{x})}\partial_T^xg(\bar{x},v)-\partial_T^xh(\bar{x},v),
 \end{align*}
which completes the proof.
\qed
\end{proof}
\begin{corollary}\label{cor01}
Consider the maximum function $\psi$ in (\ref{psi1}) in a special case where $h(x,v)=0,$ i. e., $\psi(x):=\underset{v\in V}{\max}~ g(x,v).$ Further, assume that there exists a neighbourhood $N(\bar{x})$ of $\bar{x}$ such that for each $x\in N(\bar{x})$ and $v\in V,~g$ satisfies assumptions $A_1-A_4.$ Then
\begin{equation}\label{reg1}
\partial_L\psi(\bar{x})=\hat{\partial}\psi(\bar{x})=\partial_T\psi(\bar{x})=\co\underset{v\in V(\bar{x})}{\bigcup}\partial_T^xg(x,v),
\end{equation}
which means the regularity of $\psi$ at $\bar{x}$ in the sense of Mordukhovich.
\end{corollary}
\begin{proof}
Using Theorem \ref{601}, it is easily seen that
\begin{equation}\label{reg2}
\partial_L\psi(\bar{x})\subseteq\co\underset{v\in V(\bar{x})}{\bigcup}\partial_T^xg(x,v).
\end{equation}
Then, due to \cite[theorem 2]{p2} and Lemma \ref{lem2}, we get
\begin{equation}\label{reg3} 
\co\underset{v\in V(\bar{x})}{\bigcup}\partial_T^xg(x,v)=\partial_T\psi(\bar{x})=\hat{\partial}\psi(\bar{x})=\partial_L\psi(\bar{x}).
\end{equation}
Finally, (\ref{reg2}) and (\ref{reg3}) imply (\ref{reg1}), and the proof is completed.
\qed
\end{proof}
\section{Constraint Qualifications }\label{CQ}
In this section, we focus  mainly on two  nonsmooth constraint qualifications in the face of data uncertainty based on the tangential subdifferential. Let us consider the following robust constraint system
\[
S:=\Big\{x\in\mathbb{R}^n\ |\ \psi_j(x):=\max_{v_j\in V_j}~g_j(x,v_j)\leq0,\ j\in J:=\{1,2,\ldots,m\}\Big\}.
\] 
Suppose that $\bar{x}$ is a feasible point of $S.$ Moreover, here and subsequently, we suppose that assumptions $A_1-A_4$ are also satisfied for the functions $g_j\ (j\in J).$ 
As usual in classical optimization, we also require to use the following linearized
cone at $\bar{x}:$
\begin{align*}
 G'(\bar{x})&:=\Big\{d\in\mathbb{R}^n\ |\ \psi'_j(\bar{x};d)\leq0,\ \forall j\in J(\bar{x})\Big\},
\end{align*}
where $J(\bar{x}):=\big\{j\in J~|~ \psi_j(\bar{x})=0\big\}$ is the index set of active constraints in $RP_S.$
In the following, it is important to pay our attention to introduce some new constraint qualifications in the framework of tangential subdifferential concept.
 We say that the generalized
 \begin{enumerate}
\item[$\bullet$]  Abadie constraint qualification (GACQ) holds at $\bar{x}$ if $G'(\bar{x})\subseteq T(\bar{x};S).$
\item[$\bullet$]  Error bound constraint qualification (GEBCQ) is satisfied at $\bar{x}$ if there exist positive scalars $\delta$ and $\sigma$ such that for all $x\in\mathbb{B}_{\delta}(\bar{x})\setminus S,$
\begin{equation}\label{700}
d_S(x)\leq\sigma\sum_{j\in J(\bar{x})}\psi_j^+(x),
\end{equation}
where $\psi_j^+(x):=\max\big\{\psi_j(x),0\big\},~j\in J(\bar{x}).$
\end{enumerate}
The first result of this section provides us with the relationship between the Fr\'{e}chet subdifferential of the distance function and the tangential subdifferential of the constructed functions of $\psi_j~(j\in J)$ under GEBCQ.
\begin{theorem}\label{th001}
Suppose that GEBCQ is satisfied at $\bar{x}\in S$ with some positive scalers $\delta,\sigma.$ Then the following assertion holds: 
\begin{equation}\label{equ02}
\hat{\partial}d_S(\bar{x})\subseteq\bigcup_{j\in J(\bar{x})}\co\Big(
\bigcup_{\begin{array}{l}
 \scriptstyle{\quad  \lambda_j\in[0,\sigma]}\\
   \scriptstyle{\quad v_j\in V_j(\bar{x})}\\
  \end{array}}\lambda_j\partial_T^xg_j(\bar{x},v_j)\Big).
\end{equation}
\end{theorem}
\begin{proof}
Take $\xi\in\hat{\partial}d_S(\bar{x}).$ Then for each $\varepsilon>0,$ one can find a positive real scalar $\delta'$ such that for all $x\in\mathbb{B}_{\delta'}(\bar{x})\setminus\{\bar{x}\},$
\[
-\varepsilon\leq\frac{d_S(x)-d_S(\bar{x})-\langle\xi,x-\bar{x}\rangle}{||x-\bar{x}||}.
\]
By taking $\hat{\delta}=\min\{\delta,\delta'\},$ we get for all $x\in\mathbb{B}_{\hat{\delta}}(\bar{x})\setminus\{\bar{x}\}:$
\begin{align*}
-\varepsilon & \leq\frac{d_S(x)-d_S(\bar{x})-\langle\xi,x-\bar{x}\rangle}{||x-\bar{x}||}\\
& \leq\frac{\sigma\underset{j\in J(\bar{x})}{\sum}\psi_j^+(x)-\sigma\underset{j\in J(\bar{x})}{\sum}\psi_j^+(\bar{x})-\langle\xi,x-\bar{x}\rangle}{||x-\bar{x}||},
\end{align*}
which yields immediately $\xi\in\sigma~\hat{\partial}\big(\underset{j\in J(\bar{x})}{\sum}\psi_j^+\big)(\bar{x}).$
Therefore, we get
\begin{align}
\hat{\partial}d_S(\bar{x}) & \subseteq\sigma~\hat{\partial}\big(\underset{j\in J(\bar{x})}{\sum}\psi_j^+\big)(\bar{x}) \label{align1}\\
& \subseteq\sigma~\partial_L\big(\underset{j\in J(\bar{x})}{\sum}\psi_j^+\big)(\bar{x}) \nonumber\\
& \subseteq\sigma~\underset{j\in J(\bar{x})}{\sum}\partial_L\psi_j^+(\bar{x}) \nonumber\\
& =\sigma\underset{j\in J(\bar{x})}{\sum}\co~\Big(\{0\}\cup\partial_T\psi_j(\bar{x})\Big) \label{00812**}\\
& =\sigma\underset{j\in J(\bar{x})}{\sum}\Big(\underset{\lambda_j\in[0,1]}{\bigcup}\lambda_j\partial_T\psi_j(\bar{x})\Big) \nonumber\\
& =\underset{j\in J(\bar{x})}{\sum}\Big(\underset{\lambda_j\in[0,\sigma]}{\bigcup}\lambda_j\partial_T\psi_j(\bar{x})\Big) \nonumber\\
& =\underset{j\in J(\bar{x})}{\sum}\underset{\lambda_j\in[0,\sigma]}{\bigcup}\lambda_j\Big(
\co\underset{v_j\in V_j(\bar{x})}{\bigcup}\partial_T^xg_j(\bar{x},v_j)\Big) \label{00812*}  \\
& =\underset{j\in J(\bar{x})}{\sum}\co\Big(
\bigcup_{\begin{array}{l}
 \scriptstyle{\quad  \lambda_i\in[0,\sigma]}\\
   \scriptstyle{\quad v_j\in V_j(\bar{x})}\\
  \end{array}}\lambda_j\partial_T^xg_j(\bar{x},v_j)\Big),\label{00812} 
\end{align}
where (\ref{00812**}) and (\ref{00812*}) are respectively, due to corollary \ref{cor01} and \cite[theorem 2]{p2}, which complete the proof.
\qed 
\end{proof}
The next theorem investigates the interrelations of the limiting subdifferential of the distance function and the tangential subdifferential of the constructed functions  under GEBCQ.
\begin{theorem}\label{th002}
Assume that GEBCQ  satisfies at $\bar{x}\in S$ with some positive scalers $\delta,\sigma.$ Then 
\begin{equation}\label{00813}
\partial_L d_S(\bar{x})\subseteq\sum_{j\in J(\bar{x})}\co\Big(
\bigcup_{\begin{array}{l}
 \scriptstyle{\quad  \lambda_j\in[0,\sigma]}\\
   \scriptstyle{\quad v_j\in V_j(\bar{x})}\\
  \end{array}}\lambda_j\partial_T^xg_j(\bar{x},v_j)\Big).
\end{equation}
\end{theorem}
\begin{proof}
It is known by \cite[Theorem 1.33]{boris4} that for each $\bar{x}\in S:$
\begin{equation}\label{0801}
\partial_Ld_S(\bar{x})=N_L(\bar{x};S)\cap\mathbb{B}.
\end{equation}
Now let us take $\xi\in\partial_Ld_S(\bar{x}).$ Clearly (\ref{0801}), yields  $\xi\in N_L(\bar{x};S)$ where $||\xi||\leq1.$
Hence we can find sequences $\xi^k\rightarrow\xi$ and $x^k\rightarrow\bar{x}$ such that for each $k\in\mathbb{N},~x^k\in S$ and $\xi^k\in\hat{N}(x^k;S).$ Passing to a subsequence, we can assume without loss of generality that $||\xi^k||\leq1$ and $\xi^k\in\hat{N}(x^k;S)\cap\mathbb{B}=\hat{\partial}d_S(x^k),$ (see \cite[Theorem 1.33]{boris4}).
By the similar arguments as used in Theorem \ref{th001}, it follows that
\begin{equation}\label{00810}
\xi^k\in\sum_{\begin{array}{l}
 \scriptstyle{\quad  j\in J(\bar{x})}
  \end{array}}\underset{\lambda_j\in[0,\sigma]}{\bigcup}~\lambda_j\partial_T\psi_j(x^k).
\end{equation}
There are $\xi_j^k\in\partial_T\psi_j(x^k)$ and  $\lambda_j^k\in[0,\sigma]$ such that $\xi^k=\underset{j\in J(\bar{x})}{\sum}\lambda_j^k\xi_j^k.$ 
Due to the boundedness of $\partial_T\psi_j(x^k)$ and $\lambda_j^k,$ we can assume by passing to a subsequence that $\xi_j^k\rightarrow\xi_j\in\partial_T\psi_j(\bar{x})$ and $\lambda_j^k\rightarrow\lambda_j,~j\in J(\bar{x}).$ Therefore, taking limit as $k\rightarrow\infty,$ we get
\begin{align*}
\xi & =\underset{j\in J(\bar{x})}{\sum}\lambda_j\xi_j\\
 &\in\sum_{\begin{array}{l}
 \scriptstyle{\quad  j\in J(\bar{x})}\\
   \scriptstyle{\quad \lambda_j\in[0,\sigma]}\\
  \end{array}}\lambda_j\partial_T\psi_j(\bar{x}) \\
  & =\underset{j\in J(\bar{x})}{\sum}\co\Big(
\bigcup_{\begin{array}{l}
 \scriptstyle{\quad  \lambda_i\in[0,\sigma]}\\
   \scriptstyle{\quad v_j\in V_j(\bar{x})}\\
  \end{array}}\lambda_j\partial_T^xg_j(\bar{x},v_j)\Big),
\end{align*}
where the last equality is due to (\ref{00812}). Hence (\ref{00813}) is satisfied and the proof is completed. 
\qed
\end{proof}
The last result of this section verifies that GEBCQ implies GACQ.
\begin{theorem}\label{t11}
Suppose that the GEBCQ holds at the feasible point $\bar{x}\in S.$ Then $\bar{x}$ satisfies the GACQ.
\end{theorem}
\begin{proof}
By the fact that the functions $\psi_j~(j\in J(\bar{x}))$ are tangentially convex at $\bar{x}\in S,$ it is easily seen that all the assumptions of \cite[Theorem 3.2]{MMN1} are satisfied. Hence the proof is simple and left to the reader.
\qed
\end{proof}
We conclude this section with the following illustrative examples.
The first example present a situation that both GEBCQ and GACQ are satisfied.
\begin{example}\label{exam3}
Consider the following constrained system:
\begin{equation*}\label{00801}
S:=\Big\{x=(x_1,x_2)\in\mathbb{R}^2~|~\psi_j(x):=\underset{v_j\in V_j}{\max}~g_j(x,v_j)\leq0~|~ j=1,2,3\Big\},
\end{equation*}
with
  \begin{eqnarray*}
    \begin{array}{lr}
  \displaystyle g_1(x,v_{1}):=-x_1+2v_{11}v_{12}|x_2|\leq0,\ \forall v_1=(v_{11},v_{12})\in V_1, \\
\displaystyle  g_2(x,v_2):=-(v_{21}+1)^2x_1^2-(v_{22}+1)(x_2-1)^2+1\leq0,\ \forall v_2=(v_{21},v_{22})\in V_2,\\
\displaystyle  g_3(x,v_3):=(v_{31}-1)^2x_1^2+(v_{32}-1)(x_2+1)^2+1\leq0,\ \forall v_3=(v_{31},v_{32})\in V_3,
   \end{array}
  \end{eqnarray*}
where
\begin{eqnarray*}
    \begin{array}{lr}
  \displaystyle V_1:=\big\{v_1=(v_{11},v_{12})\in\mathbb{R}^2\ |\ v_{11}^2+v_{12}^2\leq1,~ v_{11}\leq0~\mbox{or}~  v_{12}\leq0\big\}, \\
\displaystyle V_2:=[0,1]\times[0,1]~
\mbox{and}~
V_3:=[-1,0]\times[-1,0].
   \end{array}
  \end{eqnarray*}
Clearly,
\[
 \psi_1(x):=\max_{v_1\in V_1}\ g_1(x,v_1)=
\left\{
\begin{array}{ll}
-x_1+|x_2|, &\  x_2\neq0\\
-x_1,  & \  x_2 = 0,
\end{array} \right.
 \]
 \[
\psi_2(x):=\max_{v_2\in V_2}\ g_2(x,v_2)=
\left\{
\begin{array}{ll}
1, & \ x_1=0,x_2=1\\
-(x_2-1)^2+1,  & \ x_1 = 0,x_2\neq1\\
-x_1^2+1,    & \ x_1\neq0,x_2=1\\
-x_1^2-(x_2-1)^2+1, & \ x_1\neq0, x_2\neq1,
\end{array} \right.
 \]
 and
\[
\psi_3(x):=\max_{v_3\in V_3}\ g_3(x,v_3)=
\left\{
\begin{array}{ll}
1, & \ x_1=0,x_2=-1\\
-(x_2+1)^2+1,  & \ x_1 = 0,x_2\neq-1\\
-x_1^2+1,    & \ x_1\neq0,x_2=-1\\
-x_1^2-(x_2+1)^2+1, & \ x_1\neq0, x_2\neq-1.
\end{array} \right.
 \]
 Moreover,
  \[
V_1(x)=
\left\{
\begin{array}{ll}
\{(\frac{-1}{\sqrt{2}},\frac{-1}{\sqrt{2}})\}, & \ x_2\neq0\\
V_1,  &\ x_2=0,
\end{array} \right.
 \]
 \[
V_2(x)=V_{21}(x)\times V_{22}(x),
\]
where
 \[
V_{21}(x)=
\left\{
\begin{array}{ll}
[0,1], & \ x_1=0\\
\{0\},  &\ x_1\neq0,
\end{array} \right.
\qquad
V_{22}(x)=
\left\{
\begin{array}{ll}
[0,1], & \ x_2=1\\
\{0\},  &\ x_2\neq1,
\end{array} \right.
 \]
and
\[
V_3(x)=V_{31}(x)\times V_{32}(x),
 \]
where
 \[
V_{31}(x)=
\left\{
\begin{array}{ll}
[-1,0], & \ x_1=0\\
\{0\},  &\ x_1\neq0,
\end{array} \right.
\quad
V_{32}(x)=
\left\{
\begin{array}{ll}
[-1,0], & \ x_2=-1\\
\{0\},  &\ x_2\neq-1.
\end{array} \right.
\] 
It is easily seen that the functions $\psi_j(j=1,2,3)$ satisfy assumptions $A_1-A_4$ at the feasible point $\bar{x}=(0,0).$  An easy computation shows that $\partial_T \psi_1(\bar{x})=\{-1\}\times[-1,1],\ $ $\partial_T \psi_2(\bar{x})=\{(0,2)\}$ and $\partial_T \psi_3(\bar{x})=\{(0,-2)\}.$ Further, we see that for all $d=(d_1,d_2)\in\mathbb{R}^2,\ \psi'_1(\bar{x};d)=-d_1+|d_2| ,\ \psi'_2(\bar{x};d)=2d_2$
and $\psi'_3(\bar{x};d)=-2d_2.$
Hence, $T(\bar{x};S)=G'(\bar{x})=\mathbb{R}_+\times\{0\},$ which shows that GACQ  holds at $\bar{x}.$
Furthermore, it is a simple matter to see that GEBCQ is satisfied at $\bar{x}$ with $\sigma=1.$
\end{example} 
The second example indicates that GACQ is not a necessary condition for GEBCQ.
{\begin{example}\label{exa003}
Consider the following set $S$ is given by
\begin{equation*}\label{00801}
S:=\Big\{x=(x_1,x_2)\in\mathbb{R}^2|\ \psi(x):=\underset{v\in V}{\max}~g(x,v)\leq0~|~ v_1^2+v_2^2\leq1,~v_1v_2\geq0\Big\}.
\end{equation*}
 where $g(x,v)= ||(v_1,v_2)||||(x_1,x_2)||-x_2,~v=(v_1,v_2)\in V.$  Obviously $\psi$ satisfies assumptions $A_1-A_4$ at $\bar{x}=(0,0)\in S,$ for all $v\in V.$ A simple computations gives us
\[
 \psi(x)=
\left\{
\begin{array}{ll}
0, &\  (x_1,x_2)=0\\
||(x_1,x_2)||-x_2,  & \ (x_1,x_2)\neq0,
\end{array} \right.
 \]
and
 \[
V(x)=\Big\{(v_1,v_2)\in V\ |\ v_1^2+v_2^2=1,\ v_1v_2\geq0\Big\}.
 \]
 Therefore,
 $\psi'(\bar{x};d)=||(d_1,d_2)||-d_2$ and $T(S;\bar{x})=G'(\bar{x})=S=\{0\}\times\mathbb{R}_+,$
which implies that  GACQ  satisfies at $\bar{x}.$  Now let us show that GEBCQ is not satisfied at $\bar{x}.$
  Taking $x_k=(\frac{1}{k^2},\frac{1}{k}),$ it is clear  that for each arbitrary $k\in\mathbb{N},$
  $d_S(x_k)=\frac{1}{k^2}$ and 
  $\psi^+(x_k)=\sqrt{\frac{1}{k^2}+1}-\frac{1}{k}.$
  Thus one has
  \[\lim_{k\rightarrow+\infty}\frac{d_S(x_k)}{\psi^+(x_k)}=
  \lim_{k\rightarrow+\infty}\frac{\frac{1}{k^2}}{\sqrt{\frac{1}{k^4}+\frac{1}{k^2}}-\frac{1}{k}}=+\infty,\]
  which contradicts GEBCQ at $\bar{x}.$ 
\end{example}
\section{Optimality Conditions for a Robust DTC  programming problem}\label{OPT} 
The main aim of this section is to obtain the optimality conditions for $RP_S$ in terms of tangential subdifferential under GEBCQ and GACQ.  In the following, we divide our research into two parts: 
\begin{itemize} 
 \item[1.] We will 
investigate our optimality conditions for $RP_S$ in a special form, where $\psi_0$  is defined as a maximum of DTC functions.
\item[2.] We try to obtain optimality conditions for $RP_S,$ where $\psi_0$ is defined as a DTC/supremum objective function. 
\end{itemize}
\subsection{Problems with  DTC  Objective Functions under Uncertainty}\label{Sopt}
Consider $RP_S$ that is defined by
\begin{align*}
  \min &\quad\psi_0(x):=\underset{v_0\in V_0}{\max}~ g_0(x,v_0)-h_0(x,v_0), \hspace{2cm}  (RP_1)\\
  \mbox{s.t.} & \hspace{.4cm}x\in S,
    \end{align*}
where $V_0\subseteq\mathbb{R}^q,$ 
is a  nonempty compact subset of $\mathbb{R}^q$ 
 and
 $g_0,h_0:\mathbb{R}^n\times V_0\rightarrow\mathbb{R}\cup\{+\infty\}.$
The following theorem provides necessary conditions for robust optimality in $RP_1.$ 
\begin{theorem}\label{621}
Suppose that $\bar{x}\in S$ is a local minimizer of $RP_1.$  Then 
\begin{equation}\label{622}
0\in\co\underset{v_0\in V_0(\bar{x})}{\bigcup}\Big(\partial_T^xg_0(\bar{x},v_0)-\partial_T^xh_0(\bar{x},v_0)\Big)+K\partial_Ld_S(\bar{x}),
\end{equation}
where $L$ is a positive constant.
\end{theorem}
\begin{proof}
Since $\psi_0$ is a Lipschitz function, we can use the exact penalty theorem \cite[Proposition 6.3.2]{Borwein}, for a sufficiently large enough $K>0,$ and get by the local optimality at $\bar{x}:$
\begin{align*}
 0 &\in\partial_L\big(\psi_0+Kd_S\big)(\bar{x})\\
   & \in\partial_L\psi_0(\bar{x})+K\partial_Ld_S(\bar{x})\\
 & \subseteq\co\underset{v_0\in V_0(\bar{x})}{\bigcup}\Big(\partial_T^xg_0(\bar{x},v_0)-\partial_T^xh_0(\bar{x},v_0)\Big)+K\partial_Ld_S(\bar{x}),
 \end{align*}
where the last inclusion is due to Theorem \ref{601} and completes the proof.
\qed
\end{proof}
The following theorem provides necessary conditions for optimality in $RP_1$ under GEBCQ.
\begin{theorem}\label{701}
Let $\bar{x}$ be a local optimal solution for $RP_1$ and GEBCQ holds at this point. Then there exist $K>0,$ with some constant $\sigma>0$ such that
\begin{equation}\label{702}
0\in\co\underset{v_0\in V_0(\bar{x})}{\bigcup}\Big(\partial_T^xg_0(\bar{x},v_0)-\partial_T^xh_0(\bar{x},v_0)\Big)+
  \underset{j\in J(\bar{x})}{\sum}\sum_{i=1}^{n+1}
  \lambda_{ij}~\partial_T^xg_j(\bar{x},v_{ij}),
    \end{equation}
    where  $v_{ij}\in V_j(\bar{x}),~ \lambda_{ij}\geq0,~,\sum_{i=1}^{n+1}\lambda_{ij}\in[0,\sigma K].$ 
\end{theorem}
\begin{proof}
The result is a direct consequence of Theorems \ref{th002} and \ref{621}.
 \qed
\end{proof}

The next result of this section establishes the optimality conditions under  GACQ.
\begin{theorem}\label{ACQ}
   Suppose that $\bar{x}$ is a local optimal solution of $RP_1,$ and
    GACQ satisfies at this point.
    Then 
   \begin{equation}\label{0702}
0\in\co\underset{v_0\in V_0(\bar{x})}{\bigcup}\Big(\partial_T^xg_0(\bar{x},v_0)-\partial_T^xh_0(\bar{x},v_0)\Big)+
  \cl\Big(\underset{{\begin{array}{l}
 \scriptstyle{\quad v_{ij}\in V_j(\bar{x}) }\\
   \scriptstyle{\lambda_{ij}\geq0,\ j\in J(\bar{x})}\\
   \scriptstyle{\qquad l\in\mathbb{N}}\\
  \end{array}}}{\bigcup}
  \sum_{i=1}^l
  \lambda_{ij}\partial_T^xg_j(\bar{x},v_{ij})\Big).
    \end{equation}

   \end{theorem}
\begin{proof}
It is a simple matter to see that the local optimality of $\bar{x}$ yields $\psi'_0(\bar{x};d)\geq0$ for all $d\in T(\bar{x};S).$ Therefore $\bar{d}=0$ is a global minimizer of the following problem:
\begin{align}
&\min ~~~\tilde{\psi_0}(d) \label{0001'}\\
&~\mbox{s.t.}~~d\in T(\bar{x};S), \nonumber
\end{align}
where $\tilde{\psi_0}(d):=\psi'_0(\bar{x};d).$ Considering $K>0$ sufficiently large enough, (\ref{0001'}) is also equivalent to the following unconstrained optimization problem:
\begin{equation}\label{0002}
\underset{d\in\mathbb{R}^n}{\min}~~~\tilde{\psi_0}(d)+Kd_{T(\bar{x};S)}(d).
\end{equation}
Now due to Theorem \ref{621}, we get

\begin{align*}
0 & \in\partial_L(\tilde{\psi_0}+Kd_{T(\bar{x};S)})(0)\\
& \subseteq\partial_L\tilde{\psi_0}(0)+K\partial_L d_{T(\bar{x};S)}(0)\\
& \subseteq\co\underset{v_0\in V_0(\bar{x})}{\bigcup}\Big(\partial\tilde{g_0}(0,v_0)-\partial\tilde{h_0}(0,v_0)\Big)+K\partial d_{G'(\bar{x})}(0)\\
& =\co\underset{v_0\in V_0(\bar{x})}{\bigcup}\Big(\partial_T^x{g_0}(\bar{x},v_0)-\partial_T^x{h_0}(\bar{x},v_0)\Big)+K(N(0;G'(\bar{x}))\cap\mathbb{B}),
\end{align*}
where the last equality is due to Theorem \ref{621} and \cite[Theorem 1.33]{boris4}.

To obtain the desired result, we assume that $\xi\in N(0;G'(\bar{x}))\cap\mathbb{B}.$ Hence, $\xi\in N(0;G'(\bar{x}))$ with $||\xi||\leq1.$ Now by using the similar arguments as in \cite[Theorem 4.1(iii)]{MMN1} and \cite[Theorem 3(iii)]{p2}, it follows that
\[\begin{array}{ll}
N(0;G'(\bar{x}))=(G'(\bar{x}))^-&=\Big(\underset{j\in J(\bar{x})}{\bigcap}\big(\partial_T\psi_j(\bar{x})\big)^-\Big)^-\\
 & =\Big(\underset{j\in J(\bar{x})}{\bigcup}\partial_T\psi_j(\bar{x})\Big)^{--}\\
 & =\cl\Big(\mathbb{R}_+(\co\underset{v_j\in V_j(\bar{x})}{\bigcup}\partial_T^xg_j(\bar{x},v_j))\Big)\\
 & =\cl\Big(\underset{{\begin{array}{l}
 \scriptstyle{\quad v_{ij}\in V_j(\bar{x}) }\\
   \scriptstyle{\lambda_{ij}\geq0,\ j\in J(\bar{x})}\\
   \scriptstyle{\qquad l\in\mathbb{N}}\\
  \end{array}}}{\bigcup}
  \sum_{i=1}^l
  \lambda_{ij}\partial_T^xg_j(\bar{x},v_{ij})\Big),
\end{array}\]
where the forth equality is due to \cite[Theorem 33.14]{Borwein} and \cite[Theorem 2]{p2}. 
This yields inclusion (\ref{0702}) and the proof of the theorem is completed.
\qed
 \end{proof}
We conclude this subsection with some illustrative examples to ensure the validity of the results.
The next example investigates the optimality conditions under GEBCQ and GACQ.
 \begin{example}
 Consider the following robust programming problem:
 \begin{eqnarray*}
    \begin{array}{lr}
 \min\hspace{.7cm}  \psi_0(x)=\underset{v_0\in V_0}\max{}~g_0(x,v_0)-h_0(x,v_0) \\
  \ \mbox{s.t.}\hspace{.75cm}\psi_j(x)=\underset{v_j\in V_j}{\max}~g_j(x,v_j)\leq0,~j=1,2, \\
     \hspace{1.5cm} x=(x_1,x_2)\in\mathbb{R}^2,
   \end{array}
  \end{eqnarray*}
 where
 \begin{eqnarray*}
    \begin{array}{lr}
 \hspace{.7cm}  g_0(x,v_0):=|x_1|,~h_0(x,v_0):=||(v_{01},v_{02})||x_1\cos{|x_2|},\ \forall v_0=(v_{01},v_{02})\in V_0, \\
 \hspace{.75cm} g_1(x,v_{1}):=2|v_{11}v_{12}x_1|^3-x_2\leq0,\ \forall v_1=(v_{11},v_{12})\in V_1,\\
  \hspace{.75cm} g_2(x,v_2):=-(v_{21}+1)x_1^2+v_{22}|x_2|\leq0,\ \forall v_2=(v_{21},v_{22})\in V_2.
   \end{array}
  \end{eqnarray*}
with
  $V_0:=\big\{v_0=(v_{01},v_{02})\in\mathbb{R}^2\ |\  v_{01}^2+v_{02}^2\leq1,\ v_{01}v_{02}\geq0\big\},~V_1:=\big\{v_1=(v_{11},v_{12})\in\mathbb{R}^2\ |\ v_{11}^2+v_{12}^2\leq1 , v_{11}v_{12}\geq0\big\}$ and $V_2:=[0,1]\times[0,1].$
 It is a simple matter to see that
\[
 \psi_1(x):=\max_{v_1\in V_1}\ g_1(x,v_1)=
\left\{
\begin{array}{ll}
|x_1|^3-x_2, &\  x_1\neq0\\
-x_2,  & \  x_1=0,
\end{array} \right.
 \]
 \[
\psi_2(x):=\max_{v_2\in V_2}\ g_2(x,v_2)=
\left\{
\begin{array}{ll}
0, & \ x_1=0,x_2=0\\
|x_2|,  & \ x_1 = 0,x_2\neq0\\
-x_1^2,    & \ x_1\neq0,x_2=0\\
-x_1^2+|x_2|, & \ x_1\neq0, x_2\neq0.
\end{array} \right.
 \]
 Furthermore,
 \[
V_1(x)=
\left\{
\begin{array}{ll}
\{(\frac{-1}{\sqrt{2}},\frac{-1}{\sqrt{2}}),(\frac{1}{\sqrt{2}},\frac{1}{\sqrt{2}})\}, & \ x_1\neq0\\
V_1,  &\ x_1=0,
\end{array} \right.
 \]
 \[
V_2(x)=V_{21}(x)\times V_{22}(x),
\]
where
 \[
V_{21}(x)=
\left\{
\begin{array}{ll}
[0,1], & \ x_1=0\\
\{0\},  &\ x_1\neq0,
\end{array} \right.
\qquad
V_{22}(x)=
\left\{
\begin{array}{ll}
[0,1], & \ x_2=0\\
\{1\},  &\ x_2\neq0.
\end{array} \right.
 \]
Clearly, $\bar{x}=(0,0)$ is an optimal solution of the problem. Furthermore, for all $d=(d_1,d_2)\in\mathbb{R}^2,$ one has
$\psi'_1(\bar{x};d)=-d_2$ and
  $\psi'_2(\bar{x};d)=|d_2|.$
 Thus
 $ \partial_T\psi_1(\bar{x})=\{(0,-1)\}$ and
  $\partial_T\psi_2(\bar{x})=\{0\}\times[-1,1].$
  A direct computation gives us $G'(\bar{x})=T(\bar{x};S)=\mathbb{R}\times\{0\}.$ This equality especially implies GACQ at $\bar{x}.$ Moreover, GEBCQ is satisfied at $\bar{x}$ with $\sigma=1.$
 Further, one \r{that} can easily find a positive scalar $\delta$ such that for each $x\in\mathbb{B}(\bar{x};\delta),~\cos|x_2|>0,$ 
 \[
\psi_0(x)=
\left\{
\begin{array}{ll}
x_1, & \ x_1\geq0 \\
-x_1-x_1\cos|x_2|, & \ x_1<0, 
\end{array} \right.
 \]
and 
  \[
V_0(x)=
\left\{
\begin{array}{ll}
V_0, & \ x_1\geq0 \\
\{(v_{01},v_{02})\in V_0~|~||(v_{01},v_{02})||=1,~v_{01}v_{02}\geq0\}, & \ x_1<0.
\end{array} \right.
 \] 
Thus
$
\psi'_0(\bar{x},d)=
\max\big\{-2d_1,d_1\big\}
 $ 
 and
 $\partial_L\psi_0(\bar{x})=[-2,1]\times\{0\}.$
 Moreover, 
 \[
g_0(x,v_0)=
\left\{
\begin{array}{ll}
x_1, & \ x_1\geq0\\
-x_1,  &\ x_1<0,
\end{array} \right.
\qquad
h_0(x,v_0)=
\left\{
\begin{array}{ll}
0, & \ x_1\geq0\\
x_1\cos|x_2|,  &\ x_2<0.
\end{array} \right.
 \]
 Hence, 
 \[
 \partial_T^xg_0(\bar{x};v_0)-\partial_T^xh_0(\bar{x};v_0)=\co\big\{(-2,0)\big\}\cup\big\{(1,0)\big\}=[-2,1]\times\{0\}.
 \]
Clearly,
\begin{align*}
0\in & \binom{\xi_0}{0}+\lambda_1\binom{0}{-1}+\lambda_2\binom{0}{\xi_2}\\
\in & \co\underset{v_0\in V_0(\bar{x})}{\bigcup}\Big(\partial_T^xg_0(\bar{x},v_0)-\partial_T^xh_0(\bar{x},v_0)\Big)+\underset{j\in J(\bar{x})}{\sum}\sum_{i=1}^{n+1}
  \lambda_{ij}~\partial_T^xg_j(\bar{x},v_{ij}),
  \end{align*}
    where  $v_{ij}\in V_j(\bar{x}),~ \lambda_{ij}\geq0,~,\sum_{i=1}^{n+1}\lambda_{ij}\in[0,\sigma K].$ 
where  $\xi_0\in[-2,1],~ \xi_2\in[-1,1],$ and $\lambda_1,\lambda_2\geq0,$ which implies the validity of inclusions (\ref{702}) and (\ref{0702}).
\end{example}
 The last example of this part illustrates, the closure in inclusion (\ref{0702}) cannot be omitted.
\begin{example}\label{EX1}
  Consider the following robust optimization problem:
  \begin{eqnarray*}
    \begin{array}{lr}
  \min\hspace{1cm}  \displaystyle (\cos v_{01}v_{02})x_2^2-e^{\sin{|x_1|}},\ \forall v_0=(v_{01},v_{02})\in V_0 \\
  \ \mbox{s.t.}\hspace{1.2cm} x\in S,
   \end{array}
  \end{eqnarray*}
  where $S$ is the feasible set defined in Example \ref{exa003}.
  Taking $g_0(x,v_0):=(\cos v_{01}v_{02})x_2^2$ and $h_0(x,v_0):=e^{\sin{|x_1|}}$ for all $v_0=(v_{01},v_{02})\in V_0,$ then we get 
\[
 \psi_0(x):=\max_{v_0\in V_0}\ g_0(x,v_0)-h_0(x,v_0)=
\left\{
\begin{array}{ll}
x_2^2-e^{\sin{|x_1|}}, &\  x_2\neq0\\
-e^{\sin{|x_1|}},  & \ x_2=0,
\end{array} \right.
 \]
and
 \[
V_0(x)=\left\{
\begin{array}{ll}
\{(v_{01},v_{02})\in\mathbb{R}^2\ |\ v_{01}v_{02}=0\}, &\  x_2\neq0\\
V_0,  &\  x_2=0.
\end{array} \right.
 \]
 Following Example \ref{exa003}, GACQ satisfies at the local optimal point $\bar{x}=(0,0).$
  It is also easy to see that
  $\hat{\partial}{\psi_0}(\bar{x})=[-1,1]\times\{0\},\ \partial_T\psi(\bar{x})=\mathbb{B}+(0,-1)$
  and
  \[\bigcup_{\lambda\geq0}\lambda\partial_T\psi(\bar{x})=\{x\ |\ x_2<0\},\]
  which yields
  \[0\in\hat{\partial}\psi_0(\bar{x})+\cl\bigcup_{\lambda\geq0}\lambda\partial_T\psi(\bar{x}).\]
  Meanwhile,
 \[
 0\notin\hat{\partial}{\psi_0}(\bar{x})+\bigcup_{\lambda\geq0}\lambda\partial_T\psi(\bar{x}),\]
  which is because GEBCQ is not satisfied at $\bar{x}.$
Thus the closure in inclusion (\ref{0702}) cannot be omitted.

\end{example}

\subsection{ Problem with DTC/supremum objective function}\label{sec1}
In the second part of Section \ref{OPT}, we consider problem $RP_S$
with $\psi_0(x):=G(x)-H(x),$ where $G,H:\mathbb{R}^n\rightarrow\mathbb{R}\cup\{+\infty\}$ are given by $ G(x):=\underset{{v_1\in V_1}}{\sup}g(x,v_1),$ and $H(x):=\underset{{v_2\in V_2}}{\sup}h(x,v_2),$ and, $V_i\subseteq\mathbb{R}^{q_i},~q_i\in\mathbb{N},~i=1,2$ are compact subsets.
Similar to previous sections and in order to get  the desired results, we put some assumptions on the functions.
\begin{enumerate}
   \item [$(B_1)$]  $g,h$ are u.s.c. at each $(x,v_i)\in\mathbb{R}^n\times V_i,~i=1,2.$ 
    \item [$(B_2)$] $g$ and $h$ are uniformly locally Lipschitz  for $v_i\in V_i,~i=1,2$ with positive constants $K_g$ and $K_h.$ 
   \item [$(B_3)$] For each $x\in\mathbb{R}^n,~g$ and $h$ are  tangentially convex functions w.r.t. $x.$
   \item [$(B_4)$] For each $(x,d)\in\mathbb{R}^n\times\mathbb{R}^n,$ the mappings $v_1\mapsto g'_x(x,v_1;d)$ and  $v_2\mapsto h'_x(x,v_2;d)$ are u.s.c. at each $v_i\in V_i,~i=1,2.$
 \end{enumerate}
To prove the optimality conditions for $RP_S$
, it is needed to recall the definition of an isolated local minimizer from \cite{st}.
\begin{definition}
 A feasible point $\bar{x}\in S$ is called an isolated local minimum for $RP_S$ if there exist positive scalars $\varepsilon$ and $\delta$ such that for each $x\in\mathbb{B}(\bar{x};\delta)\cap S$ one has
\[G(x)-H(x)\geq G(\bar{x})-H(\bar{x})+\varepsilon||x-\bar{x}||.\]
\end{definition}
Obviously such a property is stronger than the local optimality. So, it is natural to expect that this property can lead to  stronger stationary notions.
According to \cite{p3}, a feasible point $\bar{x}\in S$ is an isolated local minimizer of problem $RP_S$ if and only if there is a positive number  $ \varepsilon $ such that  $\bar{x}$ is a local minimizer of the following DTC problem:
\begin{eqnarray*}
  \begin{array}{lr}
 \min\hspace{.9cm}
 G(x)-H_{\varepsilon}(x)\hspace{2cm}(RP_{\varepsilon}) \\
 \hspace{.1cm}\mbox{s.t.}\hspace{1cm} x\in S,
  \end{array}
\end{eqnarray*}
where $H_{\varepsilon}(x):=H(x)+\varepsilon||x-\bar{x}||.$
Obviously one has
$\partial_TH_{\varepsilon}(\bar{x})=\partial_TH(\bar{x})+\varepsilon\mathbb{B}.$
Let us now introduce our stationary notions for DTC problem $RP_S.$
Following \cite{p3}, we say that a feasible point $\bar{x}\in S$ is a
\begin{enumerate}
  \item[$\bullet$]  B-stationary point for $RP_S$ if for each $d\in T(\bar{x};S),\ G'(\bar{x};d)\geq H'(\bar{x};d).$
  \item[$\bullet$]  Fr\'{e}chet-inf-stationary (F-inf-stationary) point for $RP_S$ if
  \[\partial_TH(\bar{x})\subseteq\partial_TG(\bar{x})+\hat{N}(\bar{x};S).\]
  \item[$\bullet$] Strong B-stationary point for $RP_S$ if there exists a positive number $\varepsilon$ such that for each $d\in T(\bar{x};S):$
      \[G'(\bar{x};d)\geq H'(\bar{x};d)+\varepsilon||d||.\]
  \item[$\bullet$] Strong F-inf-stationary point for $RP_S$ if
  \[\partial_TH(\bar{x})\subseteq\Int(\partial_TG(\bar{x})+\hat{N}(\bar{x};S)).\]
\end{enumerate}
In the case that $S=\mathbb{R}^n$ and $k$ and $h$ are convex functions, the notion of F-inf-stationary reduces to the so-called inf-stationarity from  \cite{Joki}. We first try to establish some results for $RP_S$ in the simple case where $S=\mathbb{R}^n.$ To this end, we consider the following unconstrained problem $RP_2:$
\begin{eqnarray*}
  \begin{array}{lr}
 \underset{x\in\mathbb{R}^n}{\min}\hspace{.9cm}
 \psi_0(x)=g(x)-H(x)\hspace{1cm}(RP_2) 
  \end{array}
\end{eqnarray*}
The next theorem shows the relationship between local isolated optimality and strong inf-stationarity of a feasible point of $RP_2.$ 
\begin{theorem}{\cite{p3}}\label{501}
$\bar{x}\in\mathbb{R}^n$ is an isolated local minimum point for $RP_2$ if and only if it is a strong inf-stationary point. 
\end{theorem}

In the following, we introduce a new concept to approximate the tangential subdifferential.
\begin{definition}
Consider  $f:\mathbb{R}^n\rightarrow\mathbb{R}\cup\{+\infty\}$ is a   tangentially convex function at $\bar{x}\in\dom f.$  For $\varepsilon\geq0,$ the $\varepsilon$-T-subdifferential (or approximate tangential subdifferential) of $f$ at  $\bar{x}$ is given by
\begin{equation*}
\partial_T^{\varepsilon}f(\bar{x}):=\Big\{\xi\in\mathbb{R}^n~|~\langle\xi,d\rangle\leq f'(\bar{x};d)+\varepsilon||d||,~\forall d\in\mathbb{R}^n\Big\}.
\end{equation*}
\end{definition}
The special case $\varepsilon=0$ yields the classical tangential subdifferential denoted by $\partial_Tf(\bar{x}).$
The next result provides us with a relationship between $\varepsilon-$T-subdifferential and tangential subdifferential of a function.
\begin{proposition}\label{502}
Suppose that $f:\mathbb{R}^n\rightarrow\mathbb{R}\cup\{+\infty\}$ is tangentially convex at $\bar{x}\in\dom f.$ Then for all $\varepsilon\geq0,$ one has 
\[
\partial_T^{\varepsilon}f(\bar{x})=\partial_Tf(\bar{x})+\varepsilon\mathbb{B}.
\]
\end{proposition}
\begin{proof}
Obviously, $\partial_Tf(\bar{x})+\varepsilon\mathbb{B}\subseteq\partial_T^{\varepsilon}f(\bar{x}).$ To prove the converse inclusion, suppose by contradiction that there exists some $\hat{\xi}\in\partial_T^{\varepsilon}f(\bar{x})\setminus\big(\partial_Tf(\bar{x})+\varepsilon\mathbb{B}\big).$  Hence, one can find a nonzero vector $d\in\mathbb{R}^n$ such that 
$\langle\hat{\xi},d\rangle>\underset{\eta\in\partial_Tf(\bar{x})+\varepsilon\mathbb{B}}{\sup}\langle\eta,d\rangle$
 or
\[
\langle\hat{\xi},d\rangle>\underset{{\begin{array}{l}
 \scriptstyle{~ \eta_1\in\partial_Tf(\bar{x}) }\\
   \scriptstyle{~\eta_2\in\varepsilon\mathbb{B}}\\
  \end{array}}}{\sup}\langle\eta_1,d\rangle+\langle\eta_2,d\rangle.
\] 
Moreover, one can find some $\eta'_1\in\partial_Tf(\bar{x})$ 
such that $\langle\eta'_1,d\rangle=f'(\bar{x};d).$ Now defining $\eta_2:=\varepsilon\frac{d}{||d||}$ implies that $\langle\hat{\xi},d\rangle>f'(\bar{x};d)+\varepsilon||d||.$ Clearly this yields $ \hat{\xi}\notin\partial_T^{\varepsilon}f(\bar{x})$ and  completes the proof.
\qed
\end{proof}
The next result establishes necessary and sufficient optimality conditions for strong inf-stationarity.
\begin{proposition}\label{503}
$\bar{x}\in\mathbb{R}^n$ is a strong inf-stationary point for $RP_2$ if and only if there exists a positive scalar $\varepsilon$ such that for each 
$\varepsilon'<\varepsilon,
\
\partial_T^{\varepsilon'}H(\bar{x})\subseteq\partial_TG(\bar{x}).
$
\end{proposition}
\begin{proof}
The proof is straightforward and left to the reader.
\qed
\end{proof}
Next we prove necessary and sufficient conditions for isolated local optimality in the unconstrained problem $RP_2.$
\begin{theorem}\label{504}
$\bar{x}$ is an isolated local minimizer for $RP_2$ if and only if there exists some positive $\varepsilon$ such that
\begin{equation}\label{505}
\co\underset{v_2\in V_2(\bar{x})}{\bigcup}\Big(\partial_T^xh(\bar{x},v_2)+\varepsilon\mathbb{B}\Big)\subseteq\co\underset{v_1\in V_1(\bar{x})}{\bigcup}\partial_T^xg(\bar{x},v_1),
\end{equation}
where
\begin{equation}\label{506}
V_1(\bar{x}):=\Big\{v_1\in V_1\ |\ G(\bar{x})=g(\bar{x},v_1)\Big\},
\end{equation}
and
\begin{equation}\label{507}
V_2(\bar{x}):=\Big\{v_2\in V_2\ |\ H(\bar{x})=h(\bar{x},v_2)\Big\}.
\end{equation}
\end{theorem}
\begin{proof}
According to Theorem \ref{501}, $\bar{x}$ is an isolated local minimum of $RP_2$ if and only if it satisfies the 
strong inf-stationarity. The later in turn, is equivalent to 
 \begin{equation}\label{eq1*}
\partial_TH(\bar{x})+\varepsilon\mathbb{B}\subseteq\partial_TG(\bar{x}).
\end{equation}
for some $\varepsilon>0.$
Now due to constructed functions $k, h$ and using \cite[Theorem 2]{p2}, it follows that
\[
\partial_TG(\bar{x})=\co\bigcup_{v_1\in V_1(\bar{x})}\partial_T^xg(\bar{x},v_1),
\]
and
\[
\partial_TH(\bar{x})=\co\bigcup_{v_2\in V_2(\bar{x})}\partial_T^xh(\bar{x},v_2),
\]
where $V_1(\bar{x})$ and $V_2(\bar{x})$ are the same as defined in (\ref{506}) and (\ref{507}). Finally, putting all above in (\ref{eq1*}) implies inclusion (\ref{505}) and the proof is completed.
\qed
\end{proof}
In the remainder of this section, we try to establish some new results for our robust set-constrained DTC problem  $RP_S.$
\begin{theorem}
Suppose that $\bar{x}\in S$ is a local minimizer of $RP_S.$ Then 
\begin{equation}\label{0511*}
\co\underset{v_2\in V_2(\bar{x})}{\bigcup}\partial_T^xh(\bar{x},v_2)\subseteq\co\underset{v_1\in V_1(\bar{x})}{\bigcup}\Big(\partial_T^xg(\bar{x},v_1)+K\partial_Ld_S(\bar{x})\Big),
    \end{equation}
 for sufficiently large enough real $K>0.$   
\end{theorem}
\begin{proof}
According to the exact penalty theorem \cite[Proposition 6.3.2]{Borwein}, $\bar{x}$ is a local minimizer of the following unconstrained problem
\begin{eqnarray*}
  \begin{array}{lr}
 \underset{x\in\mathbb{R}^n}{\min}\hspace{.5cm}
 \big(G+Kd_S-H\big)(x)
  \end{array}
\end{eqnarray*}
 for real $K>0$ sufficiently large.
Hence, it is easily seen that all the assumptions of \cite[Theorem 4.3(ii)]{boris01}  are satisfied at $\bar{x}:$ 
\begin{align}
\hat{\partial}H(\bar{x}) & \subseteq\partial_L\big(G+Kd_S\big)(\bar{x})\nonumber\\
& \subseteq\partial_LG(\bar{x})+K\partial_Td_S(\bar{x})\label{10010*}.
\end{align}
Now due to Lemma \ref{lem2}, the tangential convex property  together with the local Lipschitz property of $H$ at $\bar{x}$ implies that $\hat{\partial}H(\bar{x})=\partial_TH(\bar{x}).$  On the other hand, one can see immediately that all the assumptions of Corollary \ref{cor01} are satisfied and then according to (\ref{reg1}) it follows  that
\[
\partial_LG(\bar{x})=\co\underset{v_1\in V_1(\bar{x})}{\bigcup}\partial_T^xg(\bar{x},v_1).
\]
Now putting all above together in (\ref{10010*}), the inclusion (\ref{0511*}) is obtained and the proof is completed.
\qed
\end{proof}

The next result of this section provides necessary conditions for $RP_S$ under GEBCQ.
\begin{theorem}\label{GEBCQ2}
Let $\bar{x}\in S$ be a local solution of $RP_S.$ 
Furthermore, assume that GEBCQ holds at this point. Then one has
\begin{equation}\label{511}
\co\underset{v_2\in V_2(\bar{x})}{\bigcup}\partial_T^xh(\bar{x},v_2)\subseteq\co\underset{v_1\in V_1(\bar{x})}{\bigcup}\partial_T^xg(\bar{x},v_1)+
  \underset{j\in J(\bar{x})}{\sum}\sum_{i=1}^{n+1}
  \lambda_{ij}~\partial_T^xg_j(\bar{x},v_{ij}),
    \end{equation}
    where  $v_{ij}\in V_j(\bar{x}),~ \lambda_{ij}\geq0,~,\sum_{i=1}^{n+1}\lambda_{ij}\in[0,\sigma K].$ 
\end{theorem}
\begin{proof}
It is easily seen that the  local optimality of $\bar{x}\in S$ in $RP_S$ implies (\ref{0511*}).
Then, using Theorem \ref{th002} and GEBCQ at $\bar{x}$ yield the existence of some positive scalers $\delta,\sigma$ such that inclusion (\ref{00813}) holds for $\partial_Ld_S(\bar{x}).$ Therefore, the proof is completed. 
\qed
\end{proof}
The last result of this section provides necessary optimality conditions for $RP_S$ under GACQ.
\begin{theorem}
Suppose that $\bar{x}\in S$ is a local minimizer of $RP_S.$ 
Moreover, assume that GACQ satisfies at this point. Then 
\begin{equation}\label{GACQ2}
\co\underset{v_2\in V_2(\bar{x})}{\bigcup}\partial_T^xh(\bar{x},v_2)\subseteq\co\underset{v_1\in V_1(\bar{x})}{\bigcup}\partial_T^xg(\bar{x},v_1)+
  \cl\Big(\underset{{\begin{array}{l}
 \scriptstyle{\quad v_{ij}\in V_j(\bar{x}) }\\
   \scriptstyle{\lambda_{ij}\geq0,\ j\in J(\bar{x})}\\
   \scriptstyle{\qquad l\in\mathbb{N}}\\
  \end{array}}}{\bigcup}
  \sum_{i=1}^l
  \lambda_{ij}\partial_T^xg_j(\bar{x},v_{ij})\Big).
  \end{equation}
\end{theorem}
\begin{proof}
By using the similar arguments as in Theorems \ref{ACQ} and          \ref{GEBCQ2}, it is clear that the local optimality of $\bar{x}\in S$ yields $\psi'_0(\bar{x};d)\geq0$ for all $d\in T(\bar{x};S).$ Therefore $\bar{d}=0$ is a global minimizer of the following problem:
\begin{align}
&\min ~~~\tilde{\psi_0}(d)=\tilde{G}(d)-\tilde{H}(d) \label{0001}\\
&~\mbox{s.t.}~~~~d\in T(\bar{x};S), \nonumber
\end{align}
where $\tilde{G}(d):=G'(\bar{x};d)$ and $\tilde{H}(d):=H'(\bar{x};d).$
Now due to \cite[Theorem 4.3(ii)]{boris01},  it follows that
\begin{equation*}\label{10010}
\hat{\partial}\tilde{H}(0)\subseteq\partial_L\tilde{G}(0)+N_L(0;T(\bar{x};S)).
\end{equation*}
Obviously, the above is equivalent to 
\begin{equation}\label{10011}
\partial\tilde{H}(0)\subseteq\partial\tilde{G}(0)+(G'(\bar{x}))^-.
\end{equation}
Finally, by the similar arguments as in Theorem \ref{621} and when the constructed functions are substituted in (\ref{10011}), inclusion (\ref{GACQ2}) is obtained which completes the proof.
\qed
\end{proof}
we conclude this section with the following illustrative example. 
\begin{example}\label{EX1}
  Consider the following robust optimization problem:
  \begin{eqnarray*}
    \begin{array}{lr}
  \min\hspace{1cm}  \psi_0(x):=G(x)-H(x) \\
  \ \mbox{s.t.}\hspace{1.2cm}x\in S,
   \end{array}
  \end{eqnarray*}
  where $S$ is the feasible set defined in Example \ref{exa003},
  and 
  \[
  G(x):=\underset{(v_{11},v_{12})\in V_1}{\sup}~2|v_{11}v_{12}|x_2^2,~~~~~H(x):=\underset{(v_{21},v_{22})\in V_2}{\sup}~\cos(v_{21}v_{22})|x_1|,
  \]
$V_1:=\big\{v_1=(v_{11},v_{12})\in\mathbb{R}^2\ |\  v_{11}^2+v_{12}^2\leq1,\ v_{11}v_{12}\geq0\big\},~V_2:=\big\{v_2=(v_{21},v_{22})\in\mathbb{R}^2\ |\ v_{21},v_{22}\in[0,\frac{\pi}{2}]\big\}.$ 
A simple computation gives us
 \begin{eqnarray*}
    \begin{array}{lr}
  G(x):=x_2^2,
 ~ V_1(x):=\big\{(v_{11},v_{12})\in V_1 ~|~ v_{11}^2+v_{12}^2=1,~ v_{11}=v_{12}=\pm\frac{1}{\sqrt{2}}\big\}, \\
  H(x):=|x_1|,~ V_2(x):=\big\{(v_{21},v_{22})\in V_2 ~|~ v_{21}v_{22}=0\big\}.
   \end{array}
  \end{eqnarray*}
Hence $\psi_0(x)=x_2^2-|x_1|.$ It is easily seen that $\bar{x}=(0,0)\in S$ is a local minimizer of the problem and assumptions $B_1-B_4$ are satisfied for the functions. Furthermore, $\partial_LG(\bar{x})=\partial_TG(\bar{x})=\{(0,0)\},$ and $\hat{\partial}H(\bar{x})=\partial_TH(\bar{x})=[-1,1]\times\{0\}.$
 On the other hand, according to Example \ref{exa003}, GACQ satisfies at $\bar{x},$
  and $ \partial_T\psi(\bar{x})=\mathbb{B}+(0,-1).$
  Hence
  \[\bigcup_{\lambda\geq0}\lambda\partial_T\psi(\bar{x})=\big\{x\ |\ x_2<0\big\}.\]
  This especially  yields
  \[
  [-1,1]\times\{0\}\subseteq\{(0,0)\}+\cl\bigcup_{\lambda\geq0}\lambda\partial_T\psi(\bar{x})
  \]
  which shows that (\ref{GACQ2}) holds true. But 
 \[
 [-1,1]\times\{0\}\notin\{(0,0)\}+\bigcup_{\lambda\geq0}\lambda\partial_T{\psi}(\bar{x}),\]
  because GEBCQ is not satisfied at $\bar{x}.$
Thus  (\ref{511}) is not satisfied and the closure in inclusion (\ref{GACQ2}) cannot be omitted.
\end{example}
\section{Conclusion}
This paper introduces novel optimality conditions for nonsmooth optimization problems incorporating uncertain data, where the objective function is formulated as a difference of two tangentially convex (DTC) functions. 
Initially, we employ nonsmooth calculus techniques, specifically tailored to our maximum function, to elucidate the relationships between Fr\'{e}chet, limiting, and tangential subdifferentials of the constituent functions.
Subsequently, leveraging generalized constraint qualifications,
such as GEBCQ and GACQ, and building upon the concept of tangential subdifferential, we derive optimality conditions for
 two problem types. These encompass scenarios where the objective function is defined as (i) the maximum of DTC functions and (ii) a difference of suprema of tangentially convex functions.
Significantly, unlike many related studies, our results are established without imposing convexity assumptions on the uncertain sets or concavity/convexity requirements on the involved functions with respect to the uncertain parameters. Consequently, the results presented herein hold true for a broad class of problems, potentially encompassing those considered in differentiable or even DC forms.
\begin{acknowledgements}
The authors thank to anonymous referees for their valuable comments and suggestions which improve the paper. Further, the authors would like to thank the Editor for the help in processing of the article.
\end{acknowledgements}
\section*{Declarations}
\hspace{-.3cm}
\textbf{ Funding.} The first-named author was supported by Iran National Science Foundation (INSF, No. 4001956) and by a grant from IPM (No. 1404490039).\\
\textbf{Conflict of interest.} 
The authors have no relevant financial or non-financial interest to disclose.\\
\textbf{Ethical approval.} This manuscript does not contain any studies with human participants or animals performed by any of the authors.\\
\textbf{Data availability. } This manuscript has no associated data.



\end{document}